\documentclass{tac}

\usepackage{amsmath}
\usepackage{amssymb}
\usepackage{amsfonts}

\usepackage{ntheorem}

\usepackage{tikz}
\usetikzlibrary{decorations.markings}
\tikzset{->-/.style={decoration={markings, mark=at position .5 with {\arrow{>}}}, postaction={decorate}}}
\tikzset{-<-/.style={decoration={markings, mark=at position .5 with {\arrow{<}}}, postaction={decorate}}}
\usetikzlibrary{shapes.geometric}

\usepackage{hyperref}

\newcommand\id{\operatorname{id}} 
\newcommand\op{\mathrm{op}} 
\newcommand\obj{\operatorname{Obj}} 
\newcommand\mor{\operatorname{Mor}} 
\newcommand\dom{\operatorname{dom}} 
\newcommand\cod{\operatorname{cod}} 
\newcommand\reflects\bowtie 

\newcommand\C{\mathcal C} 
\newcommand\D{\mathcal D} 
\newcommand\R{\mathbb R} 
\newcommand\G{\mathcal G} 

\newcommand\TelCat{\mathbf{TelCat}} 
\newcommand\SymMonCat{\mathbf{SymMonCat}} 
\newcommand\Set{\mathbf{Set}} 
\renewcommand\Game{\mathbf{Game}} 
\newcommand\Lens{\mathbf{Lens}} 

\let\epsilon\varepsilon

\title{Coherence for lenses and open games}
\author{Jules Hedges}

\begin{document}

\maketitle

\begin{abstract}
	Categories of polymorphic lenses in computer science, and of open games in compositional game theory, have a curious structure that is reminiscent of compact closed categories, but differs in some crucial ways.
	Specifically they have a family of morphisms that behave like the counits of a compact closed category, but have no corresponding units; and they have a `partial' duality that behaves like transposition in a compact closed category when it is defined.
	We axiomatise this structure, which we refer to as a `teleological category'.
	We precisely define a diagrammatic language suitable for these categories, and prove a coherence theorem for them.
	This underpins the use of diagrammatic reasoning in compositional game theory, which has previously been used only informally.
\end{abstract}

\section{Introduction}

Open games \cite{hedges15b,hedges_towards_compositional_game_theory} are a foundation for economic game theory which is strongly compositional: games are built from simple components using uniform composition operators.
More specifically, open games are the morphisms of a symmetric monoidal category in which categorical composition is a primitive form of sequential play, and the monoidal product is a primitive form of simultaneous play.
Open games can be thought of as `games open along a boundary', and in particular ordinary (extensive-form) games arise as the scalars (endomorphisms of the monoidal unit \cite{abramsky05}) in the category of open games.
Compositional game theory thus shares the same motivation and philosophy as the more general research programme of open systems \cite{fong_algebra_open_interconnected_systems}.

Being morphisms in a monoidal category, open games support a diagrammatic notation.
It was noticed early that these diagrams have a curious structure intermediate between symmetric monoidal categories and compact closed categories \cite{selinger11}: wires can bend in one direction but not the other, and some (but not all) morphisms can be rotated around a bend.
Because of their formal differences to known classes of diagrams, they have previously been used only as an informal intuition.
This paper remedies this situation by defining these diagrams precisely and proving a coherence theorem for them.

In order to state a coherence theorem we must first axiomatise the properties of open games that are visible in the diagrammatic language.
To this end we define a class of monoidal categories called \emph{teleological categories}.
Thus the main goals of this paper are, firstly, to prove that the category of open games is a teleological category, and secondly, to prove that a suitable category of diagrams modulo equivalences is equivalent to the free teleological category over a signature.

It should be noted that there is no expectation that the diagrammatic language of open games is \emph{complete} in any useful sense.
That is to say, there are games that should be considered equivalent, whose diagrams are not equivalent up to deformation.
A simple example of this is the fact that games (with pure strategies) are invariant under monotone transformations of utility functions, that is to say, only the ordering of outcomes matters.
Even the classification of $2 \times 2$ bimatrix games up to equivalence is very involved \cite{goforth_robinson_topology_2x2_games}, and there is no accepted theory of equivalences of games in general.
Current work of the author on morphisms between open games, with an associated graphical language of surface diagrams, is an attempt to work towards this.
However, it cannot be ruled out that a useful theory of equivalences of open games could be axiomatised using only the 2-dimensional language of teleological categories, with sufficient ingenuity.

As a secondary goal we identity other examples of teleological categories, of which a particularly interesting example is a certain category of \emph{lenses} \cite{foster_etal_combinators_bidirectional_tree_transformations}, or more precisely, polymorphic lenses with no lens laws imposed.
There is a surprising and close relationship between open games and lenses, which arise in programming languages and database theory, in the study of `bidirectional transformations' (or `bx' for short) \cite{foster_three_complementary_approaches_bidirectional_programming}.
(This connection was noticed by Jeremy Gibbons.)
Indeed lenses can be used to factor and simplify the definition of open games: polymorphic lenses without lens laws turn out to be equivalent to the `strategically trivial' open games \cite[section 2.2.11]{hedges_towards_compositional_game_theory}, which could also be called zero-player open games.
This connection also reveals new facts about lenses: for example, effects (or co-states, that is, morphisms into the monoidal unit) in the monoidal category of lenses turn out to be continuations \cite[section 2.2.11]{hedges_towards_compositional_game_theory}.

The possibility of using string diagrams as a syntax for lenses was mentioned in \cite{hofmann_etal_symmetric_lenses} but does not appear to have been explored further.
(They are used superficially in \cite{saleh_etal_notions_bidirectional_computation_entangled_state_monads}, for example.)
A formalisation of these diagrams, at least for the notion of polymorphic lens used in this paper, follows from the coherence theorem for teleological categories.
In the resulting graphical language, a string bending around corresponds to a degenerate lens that focusses `outside' of a data structure, into the unit type.

Regarding the name, it is possible to view variants of monoidal categories as embodying different theories of causality.
In the simplest case of circuit diagrams for symmetric monoidal categories, an edge between an earlier node $\alpha$ and a later node $\beta$ represents a causal relationship where $\beta$ is in the `causal future' of $\alpha$.
More complicated variants such as compact closed and $\dag$-categories complicate this, and introduce a sort of `quantum causality' (see for example \cite{coecke_lal_causal_categories}).
Teleological categories are intended to similarly axiomatise Aristotle's final case or `telos' \cite{sep-aristotle-causality}, which is the causality due to agents (in a very loose sense) striving towards a goal.
Teleology is a key ingredient separating social science from physical science, and this is an attempt to put it on a proper mathematical foundation.
This viewpoint and terminology were originally suggested by Viktor Winschel in the context of open games.

\paragraph{Outline of the paper.} We begin in section \ref{sec:circuit-diagrams} by recalling the diagrammatic language and coherence theorem for symmetric monoidal categories.
The next two sections discuss the motivating examples: open games in section \ref{sec:open-games} and lenses in section \ref{sec:lenses}.
Section \ref{sec:teleological-categories} abstracts these to give the definition of a teleological category, and gives other examples.
Section \ref{sec:teleological-diagrams} defines the diagrammatic language for teleological categories, and in section \ref{sec:coherence} we state and prove the coherence theorem.

\paragraph{Acknowledgements.} I would like to thank Dan Marsden for discussions about diagrams and coherence, and Jeremy Gibbons for noticing that I was redeveloping the theory of lenses from scratch under a different name, and introducing me to the relevant literature.

\section{Circuit diagrams and string diagrams}\label{sec:circuit-diagrams}

In this paper we following the naming convention of Coecke, using the term `string diagram' to refer specifically to diagrams for compact closed categories, and refer to diagrams for monoidal categories as `circuit diagrams'.
(That name is in reference to the circuit model of quantum computation.)
Extending this, we will later introduce the term `teleological diagram' to refer to the intermediate notion suitable for teleological categories.
However this distinction is purely for clarity of exposition, and it is common to refer to all of these notions as `string diagrams' when there is no danger of ambiguity.

In this section we recall the formal theory of circuit diagrams.
We refer to \cite{joyal91} for a precise definition in terms of \emph{topological graphs with boundary}.
Roughly speaking, a diagram is a graph that is smoothly embedded in some $\R \times [a, b]$ (where the first coordinate is the `space direction' and the second the `time direction').
Nodes are partitioned into the `internal nodes', which are in $\R \times (a, b)$, and the `external nodes', which are in $\R \times \{ a, b \}$ and have the additional property that they are adjacent to exactly one edge.
The internal nodes correspond to ordinary nodes in the sense of a string diagram (which will be labelled by symbols for morphisms), while the external nodes represent the intersection of an edge with a boundary of the diagram.

\begin{definition}[\cite{joyal91,selinger11}]
	A \emph{monoidal signature} (also called a tensor scheme) $\Sigma$ consists of sets $\obj (\Sigma)$, $\mor (\Sigma)$ of \emph{object symbols} and \emph{morphism symbols}, together with, for each morphism symbol $f$, a pair of words $\dom (f), \cod (f)$ over $\obj (\Sigma)$.
	If $\dom (f) = x_1 \ldots x_m$ and $\cod (f) = y_1 \ldots y_n$ then we write
	\[ f : x_1 \otimes \cdots \otimes x_m \to y_1 \otimes \cdots \otimes y_n \]
	and we write $I$ for the empty word.
	 
	A \emph{valuation} $v$ of a monoidal signature $\Sigma$ in a monoidal category $\C$, written $v : \Sigma \to \C$, consists of a choice of object $v (x)$ of $\C$ for each object symbol $x$, and a choice of morphism $v (f)$ of $\C$ for each morphism symbol $f$, such that if $f : x_1 \otimes \cdots \otimes x_m \to y_1 \otimes \cdots \otimes y_n$ then
	\[ v (f) : v (x_1) \otimes \cdots \otimes v (x_m) \to v (y_1) \otimes \cdots \otimes v (y_n) \]
\end{definition}

\begin{definition}[\cite{joyal91}]
	Let $\Sigma$ be a monoidal signature.
	A \emph{circuit diagram} (also called a \emph{progressive polarised diagram}) over $\Sigma$ is a diagram in which edges are labelled by object symbols, and internal nodes are labelled by morphism symbols so that the labels of edges adjacent to an $f$-labelled node match $\dom (f)$ and $\cod (f)$.
	If $f : x_1 \otimes \cdots \otimes x_m \to y_1 \otimes \cdots \otimes y_n$ we draw an $f$-labelled node as
	
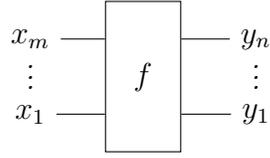
\begin{center} \begin{tikzpicture}
		\node [rectangle, minimum height=2cm, minimum width=1cm, draw] (f) at (0, 0) {$f$};
		\node (x1) at (-1.5, -.5) {$x_1$}; \node at (-1.5, .1) {$\vdots$}; \node (xm) at (-1.5, .5) {$x_m$};
		\node (y1) at (1.5, -.5) {$y_1$}; \node at (1.5, .1) {$\vdots$}; \node (yn) at (1.5, .5) {$y_n$};
		\draw [-] (x1) to (f.west |- x1); \draw [-] (xm) to (f.west |- xm);
		\draw [-] (y1) to (f.east |- y1); \draw [-] (yn) to (f.east |- yn);
	\end{tikzpicture} \end{center}
	although this is only a typographical convention, and in reality nodes are points.
	Additionally, circuit diagrams must be \emph{progressive}: Edges cannot bend around in the time-direction, which can be formalised by requiring that each edge intersects each time-slice $\R \times \{ y \}$ of the diagram at most once.
\end{definition}

It will additionally be helpful to think of external nodes as labelled by a single distinguishing symbol that is not in $\mor (\Sigma)$, say $*$.

Every internal node in a circuit diagram $A$ determines a partition of its adjacent edges into input and output edges, and determines an ordering on both classes.
We write $\alpha : e_1 \otimes \cdots \otimes e_m \to e'_1 \otimes \cdots \otimes e'_n$, where $\alpha$ refers to an internal node and $e_i, e'_j$ to its adjacent edges (and not their labels).
The diagram itself also determines an ordering on its input and output external nodes, and we write $A : e_1 \otimes \cdots \otimes e_m \to e'_1 \otimes \cdots \otimes e'_n$, where the edges are those adjacent to the external nodes.
Two diagrams over the same signature which agree in this way on their common boundary can be composed by deleting the external nodes on that boundary and linking their adjacent edges into a single edge.
Diagrams can also be composed side-by-side, with no restrictions.

Equivalences of circuit diagrams can be defined in several equivalent ways \cite{joyal91,selinger11}.
In particular, there is a `geometric' definition in terms of isotopy (or continuous deformation), and also a discrete graph-theoretic definition.
We will explicitly spell out the discrete notion, which is known as isomorphism of diagrams, in order that we can later give a variation that is suitable for teleological categories.

\begin{definition}\label{def:circuit-diagram-isomorphism}
	Let $A$ and $B$ be circuit diagrams in $\R \times [a, b]$ over a monoidal signature $\Sigma$.
	An \emph{isomorphism} $i : A \cong_c B$ is a label-preserving bijection between the nodes of $A$ and $B$, and a label-preserving bijection between the edges of $A$ and $B$, both written $i$, such that
	\begin{itemize}
		\item For each internal node $\alpha$ of $A$, $\alpha : e_1 \otimes \cdots \otimes e_m \to e'_1 \otimes \cdots \otimes e'_n$ iff $i (\alpha) : i (e_1) \otimes \cdots \otimes i (e_m) \to i (e'_1) \otimes \cdots \otimes i (e'_n)$
		
		\item For each external node $\beta$ of $A$, $\beta$ is adjacent to edge $e$ iff $i (\beta)$ is adjacent to edge $i (e)$
		
		\item $A : e_1 \otimes \cdots \otimes e_m \to e'_1 \otimes \cdots \otimes e'_n$ iff $B : i (e_1) \otimes \cdots \otimes i (e_m) \to i (e'_1) \otimes \cdots \otimes i (e'_n)$
	\end{itemize}
\end{definition}

$\cong_c$-equivalence classes of circuit diagrams over $\Sigma$ form the morphisms of a strict symmetric monoidal category $\D_c (\Sigma)$, whose objects are words over $\obj (\Sigma)$.
This category carries an obvious valuation $v_\Sigma : \Sigma \to \D_c (\Sigma)$.

\begin{theorem}[Joyal-Street coherence theorem \cite{joyal91,selinger11}]\label{joyal-street-theorem}
	Let $\Sigma$ be a monoidal signature, $\C$ a symmetric monoidal category and $w : \Sigma \to \C$ a valuation.
	Then there exists a symmetric monoidal functor $F : \D_c (\Sigma) \to \C$ such that $w = F \circ v_\Sigma$, unique up to unique monoidal natural isomorphism.
\end{theorem}

The functor $F$ is the `obvious' extension of the valuation $w$ from a signature to its category of diagrams, by computing compositionally.
The importance of this theorem is that it proves that this informal definition of $F$ is in fact well-defined, in the sense of being invariant up to equivalences of circuit diagrams.

\section{Open games}\label{sec:open-games}

Open games provide the motivating example of a teleological category.
A detailed understanding of open games is not necessary to read this paper, however, and so this section will provide a non-technical introduction.
We will also informally introduce the diagrammatic notation for open games as it has previously been used, which will be formalised in sections \ref{sec:teleological-diagrams} and \ref{sec:coherence}.
For interested readers, the most detailed exposition of open games can be found in \cite{hedges_towards_compositional_game_theory}.

Let $X, Y, R, S$ be nonempty sets.
(The requirement of nonemptiness does not appear in earlier references, but can be added harmlessly.)
By definition, an open game $\G : (X, S) \to (Y, R)$ consists of the following data:
\begin{itemize}
	\item A set $\Sigma_\G$ of \emph{strategy profiles}
	\item A \emph{play function} $\mathbf P_\G : \Sigma_\G \times X \to Y$
	\item A \emph{coplay function} $\mathbf C_\G : \Sigma_\G \times X \times R \to S$
	\item A \emph{best response} function $\mathbf B_\G : X \times (Y \to R) \to (\Sigma_\G \to \mathcal P (\Sigma_\G))$
\end{itemize}
A general open game is depicted in the diagrammatic language as
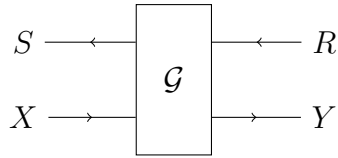
\begin{center} \begin{tikzpicture}
	\node (X) at (-2, -.5) {$X$}; \node (Y) at (2, -.5) {$Y$}; \node (R) at (2, .5) {$R$}; \node (S) at (-2, .5) {$S$};
	\node [rectangle, minimum height=2cm, minimum width=1cm, draw] (G) at (0, 0) {$\G$};
	\draw [->-] (X) to (G.west |- X); \draw [->-] (G.east |- Y) to (Y); \draw [->-] (R) to (G.east |- R); \draw [->-] (G.west |- S) to (S);
\end{tikzpicture} \end{center}

There is a symmetric monoidal category $\Game$ whose objects are pairs of nonempty sets, and whose morphisms are equivalence classes of open games after quotienting by compatible isomorphisms of strategy sets.
(More generally, open games form a symmetric monoidal bicategory whose 2-cells are compatible morphisms between strategy sets.)

We will introduce three particular families of examples.
It would take us too far afield to fully define them and discuss the intuition behind the definitions, so see \cite[sections 2.1.7 -- 2.1.9]{hedges_towards_compositional_game_theory} for this.

As a first example, there is a family of open games $\mathcal D : (X, 1) \to (Y, \mathbb R)$ called \emph{decisions}, which represent a single player making a single utility-maximising decision.
This is denoted by one of the diagrams
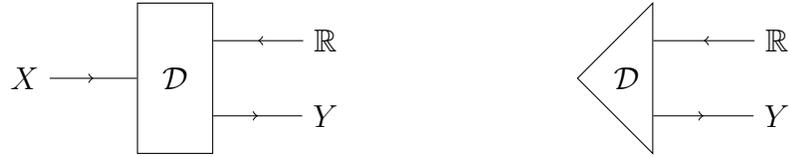
\begin{center} \begin{tikzpicture}
	\node (X) at (-2, 0) {$X$}; \node (Y1) at (2, -.5) {$Y$}; \node (R1) at (2, .5) {$\mathbb R$};
	\node [rectangle, minimum height=2cm, minimum width=1cm, draw] (D1) at (0, 0) {$\mathcal D$};
	\draw [->-] (X) to (D1); \draw [->-] (D1.east |- Y1) to (Y1); \draw [->-] (R1) to (D1.east |- R1);
	\node (Y2) at (8, -.5) {$Y$}; \node (R2) at (8, .5) {$\mathbb R$};
	\node [isosceles triangle, isosceles triangle apex angle=90, shape border rotate=180, minimum width=2cm, draw] (D2) at (6, 0) {$\mathcal D$};
	\draw [->-] (D2.east |- Y2) to (Y2); \draw [->-] (R2) to (D2.east |- R2);
\end{tikzpicture} \end{center}
with the second being the case $X = 1$, when the player makes no observation, for example in a bimatrix game.
(Triangles are commonly used to denote morphisms into or out of the monoidal unit in a diagram, but this is purely a typographical convention.)

A second family of examples are \emph{computations}, which lift functions into the category of games.
A function $f : X \to Y$, for $X$, $Y$ nonempty sets, has two associated computations: a \emph{covariant} one $f : (X, 1) \to (Y, 1)$, and a \emph{contravariant} one $f^* : (1, Y) \to (1, X)$.
These are respectively denoted
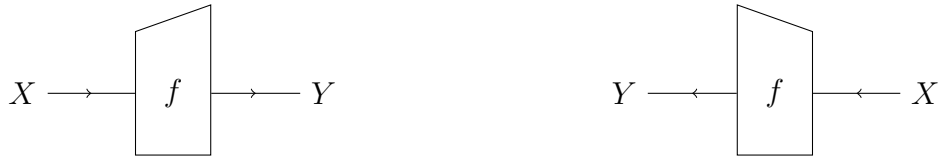
\begin{center} \begin{tikzpicture}
	\node (X1) at (0, 0) {$X$}; \node (Y1) at (4, 0) {$Y$};
	\node [trapezium, trapezium left angle=0, trapezium right angle=75, shape border rotate=90, trapezium stretches=true, minimum height=1cm, minimum width=2cm, draw] (f1) at (2, 0) {$f$};
	\draw [->-] (X1) to (f1); \draw [->-] (f1) to (Y1);
	\node (X2) at (12, 0) {$X$}; \node (Y2) at (8, 0) {$Y$};
	\node [trapezium, trapezium left angle=75, trapezium right angle=0, shape border rotate=270, trapezium stretches=true, minimum height=1cm, minimum width=2cm, draw] (f2) at (10, 0) {$f$};
	\draw [->-] (X2) to (f2); \draw [->-] (f2) to (Y2);
\end{tikzpicture} \end{center}

As the final example, for every set $X$ there is an open game $\epsilon_X : (X, X) \to (1, 1)$ called a \emph{counit}, which is denoted by
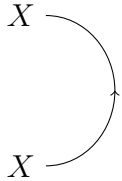
\begin{center} \begin{tikzpicture}
	\node (X1) at (0, 0) {$X$}; \node (X2) at (0, 2) {$X$};
	\draw [->-] (X1) to [out=0, in=-90] (1.25, 1) to [out=90, in=0] (X2);
\end{tikzpicture} \end{center}
Covariant computations, contravariant computations and counits are related by the \emph{counit law} \cite[sections 2.2.13 and 2.3.6]{hedges_towards_compositional_game_theory}, which states that for every function $f : X \to Y$, the diagram
\begin{center} \begin{tikzpicture}[node distance=3cm, auto]
	\node (A) {$(X, Y)$}; \node (B) [right of=A] {$(Y, Y)$};
	\node (C) [below of=A] {$(X, X)$}; \node (D) [below of=B] {$(1, 1)$};
	\draw [->] (A) to node {$f \otimes (1, Y)$} (B); \draw [->] (A) to node {$(X, 1) \otimes f^*$} (C);
	\draw [->] (B) to node {$\epsilon_Y$} (D); \draw [->] (C) to node {$\epsilon_X$} (D);
\end{tikzpicture} \end{center}
commutes.
In diagrammatic notation, this is the equation
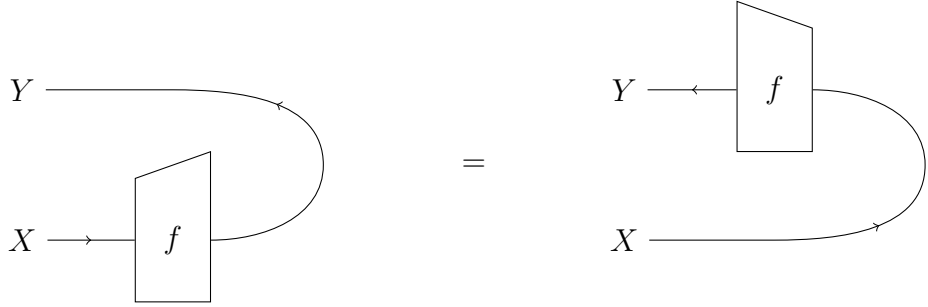
\begin{center} \begin{tikzpicture}
	\node (X1) at (0, 0) {$X$}; \node (Y1) at (0, 2) {$Y$};
	\node [trapezium, trapezium left angle=0, trapezium right angle=75, shape border rotate=90, trapezium stretches=true, minimum height=1cm, minimum width=2cm, draw] (f1) at (2, 0) {$f$};
	\draw [->-] (X1) to (f1); \draw [->-] (f1) to [out=0, in=-90] (4, 1) to [out=90, in=0] (2, 2) to (Y1);
	\node at (6, 1) {$=$};
	\node (X2) at (8, 0) {$X$}; \node (Y2) at (8, 2) {$Y$};
	\node [trapezium, trapezium left angle=75, trapezium right angle=0, shape border rotate=270, trapezium stretches=true, minimum height=1cm, minimum width=2cm, draw] (f2) at (10, 2) {$f$};
	\draw [->-] (X2) to (10, 0) to [out=0, in=-90] (12, 1) to [out=90, in=0] (f2); \draw [->-] (f2) to (Y2);
\end{tikzpicture} \end{center}

A \emph{scalar} in a monoidal category is an endomorphism of the monoidal unit \cite{abramsky05}.
The monoidal unit in the category of open games is $(1, 1)$, and therefore a scalar $\G$ consists of a set $\Sigma_\G$ and a best response function $\mathbf B_\G : \Sigma_\G \to \mathcal P (\Sigma_\G)$.

Such scalars describe a useful description of \emph{games} in game theory, where $\mathbf B_\G (\sigma)$ is the set of all rational \emph{unilateral deviations} $\sigma'$ from a strategy profile $\sigma$.
A unilateral deviation is a strategy profile $\sigma'$ that differs from $\sigma$ in the strategy of at most one player $i$, such that the payoff for player $i$ after playing $\sigma'$ is at least as large as that from playing $\sigma$.
Fixpoints of the best response function (in the sense that $\sigma \in \mathbf B_\G (\sigma)$) are called \emph{Nash equilibria}, and are strategies with the property that no rational player has incentive to unilaterally deviate from playing $\sigma$; Nash equilibria are `stable' or `self-confirming' strategies.
The other values of $\mathbf B_\G (\sigma)$ are called \emph{off-equilibrium best responses}, and are also characteristic of a game.
Although there is no generally accepted notion of equivalence of games, equality of best response functions provides a useful notion that works `across theories', for example between open games and the classical formalism of von Neumann \cite{vonneumann44}.

By composing open games from the three families we have introduced using categorical composition and tensor product, we can build scalars with the same best response functions as various standard games.
We will give three examples, each of which has two players who choose from sets $X$ and $Y$ respectively, with payoffs given by a function $u : X \times Y \to \mathbb R^2$.
Figure \ref{fig:example-games}(a) shows a game where two players make the choices simultaneously (a well-known example being rock-paper-scissors).

Figure \ref{fig:example-games}(b) shows a `game of perfect information', where the second player observes the first player's choice before making her own choice.
(An example of a game of perfect information is chess, although it has more than two stages.)
Notice the use of a comonoid operator $X \to X \otimes X$, which is the copy function lifted as a covariant computation, to copy the choice of the first player, which is both observed by the second player and used by the utility function.

Figure \ref{fig:example-games}(c) is a `game of imperfect information', intermediate between the previous two examples (an example being poker, which has both visible and private information).
There is an equivalence relation on the choices made by the first player, and the second player observes only the equivalence class.
(The function $\pi_\sim : X \to X / \sim$ is the projection onto equivalence classes.)

\begin{figure}\label{fig:example-games}
	\begin{center}
		\begin{center} \begin{tikzpicture}
			\node [isosceles triangle, isosceles triangle apex angle=90, shape border rotate=180, minimum width=2cm, draw] (D1) at (0, 0) {$\mathcal D_1$};
			\node [isosceles triangle, isosceles triangle apex angle=90, shape border rotate=180, minimum width=2cm, draw] (D2) at (0, 3) {$\mathcal D_2$};
			\node [trapezium, trapezium left angle=0, trapezium right angle=75, shape border rotate=90, trapezium stretches=true, minimum height=1cm, minimum width=2cm, draw] (U) at (3, 1.5) {$u$};
			\node (d1) at (0, -.5) {}; \node (d2) at (0, .5) {}; \node (d3) at (0, 2.5) {}; \node (d4) at (0, 3.5) {}; \node (d5) at (0, 1) {}; \node (d6) at (0, 2) {};
			\draw [->-] (D1.east |- d1) to [out=0, in=180] node [above, near end] {$X$} (U.west |- d5);
			\draw [->-] (D2.east |- d3) to [out=0, in=180] node [above] {$Y$} (U.west |- d6);
			\draw [->-] (U.east |- d5) to [out=0, in=90] (4, .5) to [out=-90, in=0] node [below] {$\mathbb R$} (D1.east |- d2);
			\draw [->-] (U.east |- d6) to [out=0, in=-90] (4, 2.5) to [out=90, in=0] node [above] {$\mathbb R$} (D2.east |- d4);
		\end{tikzpicture} \end{center}
		(a) Simultaneous game
	
		\begin{center} \begin{tikzpicture}[yscale=-1]
			\node [isosceles triangle, isosceles triangle apex angle=90, shape border rotate=180, minimum width=2cm, draw] (D1) at (2, 0) {$\mathcal D_1$};
			\node [circle, scale=.5, fill=black, draw] (m1) at (4, .5) {};
			\node [rectangle, minimum height=2cm, minimum width=1cm, draw] (D2) at (6.5, 0) {$\mathcal D_2$};
			\node [trapezium, trapezium left angle=0, trapezium right angle=75, shape border rotate=90, trapezium stretches=true, minimum height=1cm, minimum width=2cm, draw] (U) at (9, 1) {$u$};
			\node (d1) at (0, -.5) {}; \node (d2) at (0, 1.5) {};
			\draw [->-] (D1.east |- m1) to node [above] {$X$} (m1);
			\draw [->-] (m1) to [out=45, in=180] node [above, very near end] {$X$} (U.west |- d2);
			\draw [->-] (m1) to [out=-45, in=180] (D2);
			\draw [->-] (D2.east |- m1) to node [above] {$Y$} (U.west |- m1);
			\draw [->-] (U.east |- m1) to [out=0, in=90] (10, 0) to [out=-90, in=0] node [above, near end] {$\R$} (D2.east |- d1);
			\draw [->-] (U.east |- d2) to [out=0, in=90] (11, 0) to [out=-90, in=0] node [above, near end] {$\R$} (6.5, -1.5) to [out=180, in=0] (D1.east |- d1);
		\end{tikzpicture} \end{center}
		(b) Sequential game of perfect information
	
		\begin{center} \begin{tikzpicture}[yscale=-1]
			\node [isosceles triangle, isosceles triangle apex angle=90, shape border rotate=180, minimum width=2cm, draw] (D1) at (0, 0) {$\mathcal D_1$};
			\node [circle, scale=.5, fill=black, draw] (m1) at (2, .5) {};
			\node [trapezium, trapezium left angle=0, trapezium right angle=75, shape border rotate=90, trapezium stretches=true, minimum height=1cm, minimum width=2cm, draw] (p) at (4, 0) {$\pi_\sim$};
			\node [rectangle, minimum height=2cm, minimum width=1cm, draw] (D2) at (6.5, 0) {$\mathcal D_2$};
			\node [trapezium, trapezium left angle=0, trapezium right angle=75, shape border rotate=90, trapezium stretches=true, minimum height=1cm, minimum width=2cm, draw] (U) at (9, 1) {$u$};
			\node (d1) at (0, -.5) {}; \node (d2) at (0, 1.5) {};
			\draw [->-] (D1.east |- m1) to node [above] {$X$} (m1);
			\draw [->-] (m1) to [out=45, in=180] node [above, very near end] {$X$} (U.west |- d2);
			\draw [->-] (m1) to [out=-45, in=180] (p);
			\draw [->-] (p) to node [above] {$X / {\sim}$} (D2);
			\draw [->-] (D2.east |- m1) to node [above] {$Y$} (U.west |- m1);
			\draw [->-] (U.east |- m1) to [out=0, in=90] (10, 0) to [out=-90, in=0] node [above, near end] {$\R$} (D2.east |- d1);
			\draw [->-] (U.east |- d2) to [out=0, in=90] (11, 0) to [out=-90, in=0] node [above, near end] {$\R$} (6.5, -2) to [out=180, in=0] (D1.east |- d1);
		\end{tikzpicture} \end{center}
		(c) Sequential game of imperfect information
	
		\caption{Examples of scalar open games}
	\end{center}
\end{figure}

In each of these examples (and in the generalisation of them to any finite number of players), the best response function of the scalar is equal to the best response function for ordinary Nash equilibrium in standard game theory.
In this sense, we have the expressive power of ordinary game theory, but in a compositional setting.

\section{Lenses}\label{sec:lenses}

In this section we construct a category $\Lens$ whose objects are pairs $(X, S)$ of nonempty sets, and whose morphisms are called \emph{lenses}.
The category $\Lens$ has several features in common with $\Game$ and will provide a second important example of a teleological category.

\begin{definition}
	Let $X, S, Y, R$ be nonempty sets.
	A lens $\lambda : (X, S) \to (Y, R)$ consists of a pair of functions $\lambda = (v_\lambda, u_\lambda)$ where $v_\lambda : X \to Y$ and $u_\lambda : X \times R \to S$.
\end{definition}

We think of these types as follows: $X$ is the type of some data structure, and we can use the lens $\lambda$ to `zoom in' on some part of the data structure, which has type $Y$.
This is done by the function $v_\lambda$, which is mnemonic for `view'.
The data that we are viewing can then be updated, so that it now has type $R$, and the lens propagates this update to the aggregate whole, whose type changes to $S$.
This is done by the function $u_\lambda$ (for `update'), which takes the original data structure and the updated value, and returns the new state of the data structure.
As a simple example, consider a lens $\lambda$ that zooms in on the first coordinate of a pair of integers, and allows replacing it with a boolean instead.
This lens has type $\lambda : (\mathbb Z \times \mathbb Z, \mathbb B \times \mathbb Z) \to (\mathbb Z, \mathbb B)$, and is given by $v_\lambda (z_1, z_2) = z_1$ and $u_\lambda ((z_1, z_2), b) = (b, z_2)$.

The identity lens $\id_{(X, S)} : (X, S) \to (X, S)$ is given by $v_{\id_{(X, S)}} = \id_X : X \to X$ and $u_{\id_{(X, S)}} = \pi_2 : X \times S \to S$.
The composition of $\lambda : (X, S) \to (Y, R)$ and $\mu : (Y, R) \to (Z, Q)$ is given by $v_{\mu \circ \lambda} = v_\mu \circ v_\lambda$ and $u_{\mu \circ \lambda} (x, q) = u_\lambda (x, u_\mu (v_\lambda (x), q))$.

\begin{proposition}
	With these definitions, $\Lens$ is a category.
\end{proposition}

The description of lenses given here is essentially that of `concrete lenses' in \cite{pickering_gibbons_wu_profunctor_optics}, which differs slightly from usual presentation of lenses.
(The only difference between their definition and ours is the requirement that the sets are nonempty, which will be discussed later.)
The simplest lenses are `monomorphic' lenses, which do not allow updates to change the type of data.
The category of monomorphic lenses is the full subcategory of $\Lens$ whose objects have the form $(X, X)$.
Lenses are also usually considered to satisfy axioms called the lens laws, that specify intuitive properties of bidirectional transformations on data; lenses are classified as `well-behaved' or `very well-behaved', with most theoretical work focussing on the very well-behaved case.
These form a hierarchy of subcategories of $\Lens$ \cite{johnson_rosebrugh_spans_lenses}.
Lenses which can change types are generally considered polymorphic (that is to say, they can change to any type), and exist in a type system with parametric polymorphism, in which case the lens laws imply certain constraints on the types \cite[part 4]{kazak-lenses-over-tea}.
In this formulation, we consider lenses which in general not only fail to satisfy the lens laws, but for which the lens laws do not even type-check.\footnote{Cezar Ionescu has suggested calling these `outlaw lenses'.}
We choose this formulation because of its application to game theory; to the author's knowledge this is the first application of lenses that do not represent bidirectional transformations of data.

The category of lenses can be given a symmetric monoidal structure, with monoidal unit $I = (1, 1)$.
The monoidal product on objects is $(X, S) \otimes (X', S') = (X \times X', S \times S')$.
On lenses it is given by $v_{\lambda \otimes \mu} (x, x') = (v_\lambda (x), v_\mu (x'))$ and $u_{\lambda \otimes \mu} ((x, x'), (r, r')) = (u_\lambda (x, r), u_\mu (x', r'))$.
(In Edward Kmett's popular \texttt{Control.Lens} library for Haskell, the monoidal product is called \texttt{alongside}.)

\begin{proposition}
	With these definitions, $\Lens$ is a symmetric monoidal category.
\end{proposition}

This can either be proven directly, or using the following more abstract machinery.
Viewing lenses constitutes a monoidal fibration $v : \Lens \to \Set_1$, where $\Set_1$ is the category of nonempty sets.
(Note that the nonemptiness requirement is unnecessary in this paragraph, and indeed we can build a category of lenses $\Lens (\C)$ over any category $\C$ with finite products, obtaining a monoidal fibration $v : \Lens (\C) \to \C$; nonemptiness will be used later.)
This is reminiscent of \cite{johnson_etal_lenses_fibrations_universal_transformations}, but the following additional observations seem to be new.
The fibration $v$ is the opposite (in the fibrewise sense) of another that arises in a different context, namely, Jacobs' `simple fibration' $s (\Set_1)$ \cite[section 1.3]{jacobs-categorical-logic-type-theory}.
The fibre in $\Lens$ over $X$ is $\operatorname{co-kl} (X \times)^\op$, the opposite of the co-kleisli category of the left-multiplication-by-$X$ comonad (sometimes called the `reader comonad').
The entire structure can be built starting from the fact that $(X \times)$ is an indexed commutative comonad, or even more basically, from commutative comonoids using \cite[proposition 2.1]{pavlovic_geometry_abstraction_quantum_computation}, 
and the symmetric monoidal structure on $\Lens$ can be obtained from the fibres using a Grothendieck construction \cite[theorem 12.7]{shulman_framed_bicategories_monoidal_fibrations}.

This abstract point of view also serves to clarify another connection (implicit in \cite[section 2.2.14]{hedges_towards_compositional_game_theory}), between lenses and de Paiva's intuitionistic dialectica categories \cite{depaiva91a}, using \cite[exercise 1.10.11]{jacobs-categorical-logic-type-theory}.
However, the additional abstraction is unnecessary in this paper, so we will continue with the concrete presentation.

Next, we observe that there are lenses corresponding to computations and counits in the category of open games, and they satisfy a similar `counit law'.

\begin{definition}
	An \emph{adaptor} is a lens $\lambda : (X, S) \to (Y, R)$ whose update function factorises as $X \times R \xrightarrow{\pi_2} R \to S$.
\end{definition}

As bidirectional transformations, adaptors convert between data representations in a reversible way.
Adaptors are usually called `isos', but are only guaranteed to be categorical isomorphisms if we restrict to well-behaved lenses, so we follow the terminology of \cite{pickering_gibbons_wu_profunctor_optics} for clarity.

\begin{proposition}
	Adaptors form a wide symmetric monoidal subcategory $\Lens_d$.
\end{proposition}

\begin{proof}
	It is straightforward to check that $\Lens_d$ is closed under composition and monoidal product, and that identities and structure morphisms are adaptors.
\end{proof}

Given functions $f : X \to Y$ and $g : R \to S$, we write $(f, g) : (X, S) \to (Y, R)$ for the unique adaptor satisfying $v_{(f, g)} = f$ and $u_{(f, g)} = g \circ \pi_2$.
This is where we use the fact that our sets are nonempty, in order to guarantee uniqueness.
(It is possible to develop the theory without requiring nonemptiness. This leads to a definition of teleological categories in which dualisable morphisms form a separate category, with a not-necessarily-faithful `embedding'. The reason we do not do this is that it seriously complicates the diagrammatic language by requiring functorial boxes \cite{mellies_functorial_boxes_string_diagrams}, and compositional game theory has so far never required a use of the empty set. On the other hand there are game-theoretic reasons to study colimits in $\Lens$, and so this may need to change in the future.)

\begin{proposition}
	There is an isomorphism of symmetric monoidal categories $\Lens_d \cong \Set_1 \times \Set_1^\op$.
\end{proposition}

\begin{proof}
	It can be shown that $(-, -) : \Set_1 \times \Set_1^\op \to \Lens_d$ is a symmetric monoidal functor, and is invertible.
\end{proof}

\begin{proposition}
	There is a symmetric monoidal functor $-^* : \Lens_d^\op \to \Lens_d$ defined by $(X, S)^* = (S, X)$ and $(f, g)^* = (g, f)$.
\end{proposition}

\begin{proof}
	This is trivially a symmetric monoidal functor $\Set_1^\op \times \Set_1 \to \Set_1 \times \Set_1^\op$.
\end{proof}

For any object $(X, S)$ we have $(X, S) \otimes (X, S)^* = (X \times S, S \times X)$.
There is a `counit' lens $\epsilon_{(X, S)} : (X \times S, S \times X) \to (1, 1)$ given by the unique $v_{\epsilon_{(X, S)}} : X \times S \to 1$, and $u_{\epsilon_{(X, S)}} ((x, s), *) = (s, x)$.
This can be thought of intuitively as a lens that focuses `outside' a data structure into nothing, which is represented by $1$.

\begin{proposition}\label{prop:lens-counit-law}
	For each adaptor $(f, g) : (X, S) \to (Y, R)$ the diagram
	\begin{center} \begin{tikzpicture}[node distance=3cm, auto]
		\node (A) at (0, 0) {$(X \times R, S \times Y)$}; \node (B) at (6, 0) {$(Y \times R, R \times Y)$};
		\node (C) [below of=A] {$(X \times S, S \times X)$}; \node (D) [below of=B] {$(1, 1)$};
		\draw [->] (A) to node {$(f, g) \otimes \id_{(R, Y)}$} (B); \draw [->] (B) to node {$\epsilon_{(Y, R)}$} (D);
		\draw [->] (A) to node {$\id_{(X, S)} \otimes (g, f)$} (C); \draw [->] (C) to node {$\epsilon_{(X, S)}$} (D);
	\end{tikzpicture} \end{center}
	commutes.
\end{proposition}

\begin{proof}
	The view functions of the two composed lenses are both the unique function $X \times R \to 1$.
	For the update functions we check by direction calculation:
	\begin{align*}
		&u_{\epsilon_{(Y, R)} \circ ((f, g) \otimes \id_{(R, Y)})} ((x, r), *) \\
		=\ &u_{(f, g) \otimes \id_{(R, Y)}} ((x, r), u_{\epsilon_{(Y, R)}} (v_{(f, g) \otimes \id_{(R, Y)}} (x, r), *)) \\
		=\ &u_{(f, g) \otimes \id_{(R, Y)}} ((x, r), u_{\epsilon_{(Y, R)}} ((v_{(f, g)} (x), v_{\id_{(R, Y)}} (x, r)), *)) \\
		=\ &u_{(f, g) \otimes \id_{(R, Y)}} ((x, r), u_{\epsilon_{(Y, R)}} ((f (x), r), *)) \\
		=\ &u_{(f, g) \otimes \id_{(R, Y)}} ((x, r), (r, f (x))) \\
		=\ &(u_{(f, g)} (x, r), u_{\id_{(R, Y)}} (r, f (x))) \\
		=\ &(g (r), f (x)) \\
		=\ &(u_{\id_{(X, S)}} (x, g (r)), u_{(g, f)} (r, x)) \\
		=\ &u_{\id_{(X, S)} \otimes (g, f)} ((x, r), (g (r), x)) \\
		=\ &u_{\id_{(X, S)} \otimes (g, f)} ((x, r), u_{\epsilon_{(X, S)}} ((x, g (r)), *)) \\
		=\ &u_{\id_{(X, S)} \otimes (g, f)} ((x, r), u_{\epsilon_{(X, S)}} ((v_{\id_{(X, S)}} (x), v_{(g, f)} (r)), *)) \\
		=\ &u_{\id_{(X, S)} \otimes (g, f)} ((x, r), u_{\epsilon_{(X, S)}} (v_{\id_{(X, S)} \otimes (g, f)} (x, r), *)) \\
		=\ &u_{\epsilon_{(X, S)} \circ (\id_{(X, S)} \otimes (g, f))} ((x, r), *)
	\end{align*}
	Since the two lenses have equal view and update functions, they are equal.
\end{proof}

It was observed by Jeremy Gibbons that the play and coplay functions of an open game $\G : (X, S) \to (Y, R)$ can be rewritten as a single function $\Sigma_\G \to \hom_\Lens ((X, S), (Y, R))$, and the composition and tensor of open games can be written in terms of composition and tensor of lenses.
However the connections between open games and lenses run deeper, by observing that there are natural isomorphisms $X \cong \hom_\Lens ((1, 1), (X, S))$ and $Y \to R \cong \hom_\Lens ((Y, R), (1, 1))$, allowing the best response function of an open game to also be expressed in terms of lenses.
These considerations lead to a definition of morphisms between open games, which is however tangential to the topic of this paper.

\section{Teleological categories}\label{sec:teleological-categories}

We will now abstract the examples from the previous two sections, the categories of open games and lenses, into the axioms of a teleological category.
These axioms capture the features of those categories that are visible in the graphical language: a wide subcategory of dualisable objects, a dualisation functor on that subcategory, and a family of `counit' morphisms that interacts with duals via a counit law.

\begin{definition}
	A \emph{teleological category} is a symmetric monoidal category $\C$, together with:
	\begin{itemize}
		\item A wide symmetric monoidal subcategory $\C_d$ of \emph{dualisable morphisms}, with an involutive symmetric monoidal functor $-^* : \C_d^\op \to \C_d$
		
		\item An extranatural family of morphisms $\epsilon_X : X \otimes X^* \to I$ in $\C$, called \emph{counits}; more precisely, an extranatural transformation $\epsilon : F \to G$ between $F, G : \C_d^\op \times \C_d \to \C$, where $F (S, X) = X \otimes S^*$ and $G (S, X) = I$.
	\end{itemize}
	such that the diagrams
	\begin{center} \begin{tikzpicture}[node distance=2cm, auto]
		\node (A) at (0, 0) {$X^* \otimes X$}; \node (B) at (3, 0) {$X \otimes X^*$}; \node (C) [below of=B] {$I$};
		\draw [->] (A) to node {$\sigma_{X^*, X}$} (B); \draw [->] (B) to node {$\epsilon_X$} (C); \draw [->] (A) to node {$\epsilon_{X^*}$} (C);
	\end{tikzpicture} 

	\begin{tikzpicture}[node distance=2cm, auto]
		\node (A) at (0, 0) {$X \otimes Y \otimes X^* \otimes Y^*$}; \node (B) at (7, 0) {$X \otimes X^* \otimes Y \otimes Y^*$}; \node (C) [below of=B] {$I$};
		\draw [->] (A) to node {$X \otimes \sigma_{X^*, Y} \otimes Y^*$} (B); \draw [->] (B) to node {$\epsilon_X \otimes \epsilon_Y$} (C); \draw [->] (A) to node {$\epsilon_{X \otimes Y}$} (C);
	\end{tikzpicture} \end{center}
	commute for all objects $X$ and $Y$.
\end{definition}

We will now unpack parts of this definition.
Every object $X$ of $\C$ has a chosen `dual' $X^*$ satisfying $(X^*)^* = X$, $I^* = I$ and $(X \otimes Y)^* = X^* \otimes Y^*$.
(In all of our examples these equalities will be strict, although the monoidal structure will generally not be strict.)
Certain morphisms $f : X \to Y$, namely the dualisable morphisms, have a chosen `dual' $f^* : Y^* \to X^*$, functorially.
Since the category of dualisable morphisms is a symmetric monoidal subcategory, the structure morphisms are symmetries are required to be dualisable, and their duals are the structure morphisms and symmetries for the dual objects.

Every object $X$ has a chosen counit morphism $\epsilon_X : X \otimes X^* \to I$, which interact with dual objects and tensor products according to the two stated axioms.
Typically the counit morphisms will not be dualisable.
(There is an alternative definition possible, in which $\epsilon$ is more straightforwardly a monoidal extranatural transformation, and so can be composed without involving the symmetry, but at the expense of twisting the dual of a monoidal product to $(X \otimes Y)^* = Y^* \otimes X^*$, as in the definition of a compact closed category \cite{kelly80}.
This leads to a diagrammatic language in which duality is rotation rather than reflection, similar to \cite[section 4.6.2]{coecke_kissinger_picturing_quantum_processes}.)

Finally, the condition that $\epsilon$ is extranatural is a general `counit law', stating that for every dualisable morphism $f : X \to Y$, the following diagram commutes:
\begin{center} \begin{tikzpicture}[node distance=3cm, auto]
	\node (A) {$X \otimes Y^*$}; \node (B) [right of=A] {$Y \otimes Y^*$};
	\node (C) [below of=A] {$X \otimes X^*$}; \node (D) [below of=B] {$I$};
	\draw [->] (A) to node {$f \otimes Y^*$} (B); \draw [->] (A) to node {$X \otimes f^*$} (C);
	\draw [->] (B) to node {$\epsilon_Y$} (D); \draw [->] (C) to node {$\epsilon_X$} (D);
\end{tikzpicture} \end{center}

\begin{proposition}
	$\Lens$ is a teleological category.
\end{proposition}

\begin{proof}
	This is proven directly by results in section \ref{sec:lenses}.
\end{proof}

To view $\Game$ as a teleological category we must do slightly more work.
We define $\Game_d$ as the wide subcategory of $\Game$ consisting of open games $\G$ with the following properties:
\begin{itemize}
	\item $\Sigma_\G = 1 = \{ * \}$, and $* \in \mathbf B_\G (h, k) (*)$ for all $h, k$
	
	\item $\mathbf C_\G : 1 \times X \times R \to S$ is a constant function in $X$
\end{itemize}
Equivalently, $\Game_d$ can be defined as the wide subcategory consisting of open games of the form
\[ (X, S) \cong (X \times 1, 1 \times S) \xrightarrow{f \otimes g^*} (Y \times 1, 1 \times R) \cong (Y, R) \]
for functions $f : X \to Y$, $g : R \to S$.

Given an object $(X, S)$, the counit $\epsilon_{(X, S)} : (X, S) \otimes (X, S)^* = (X \times S, S \times X) \to (1, 1)$ is the evident strategically trivial open game with $\mathbf C_{\epsilon_{(X, S)}} (*, (x, s), *) = (s, x)$.
Equivalently, $\epsilon_{(X, S)}$ can be defined as
\[ (X \times S, S \times X) \cong ((X \times S) \times 1, 1 \times (S \times X)) \xrightarrow{\id_{X \times S} \otimes \sigma_{X, S}^*} \]
\[ ((X \times S) \times 1, 1 \times (X \times S)) \cong (X \times S, X \times S) \xrightarrow{\epsilon_{X \times S}} (1, 1) \]

\begin{proposition}
	With these definitions, $\Game$ is a teleological category.
\end{proposition}

\begin{proof}
	It remains to prove extranaturality of $\epsilon$.
	The proof is essentially identical to that of proposition \ref{prop:lens-counit-law} for lenses.
\end{proof}

The graphical calculus for teleological categories defined in the next section is directly inspired by the graphical calculus of compact closed categories.
Unfortunately it is not the case that every compact closed category can be seen as a teleological category, because the duality in a compact closed category is not necessarily a monoidal functor, and because the counit morphisms are not necessarily extranatural.
However, in some cases it is possible to define a duality that satisfies the axioms of a teleological category.

\begin{proposition}
	Let $\mathbf{Rel}$ be the symmetric monoidal category of sets and relations, with cartesian product of sets as the monoidal product.
	Then $\mathbf{Rel}$ is a teleological category with $\mathbf{Rel}_d = \mathbf{Rel}$, where the dual of an object is $X^* = X$, the dual of a morphism is the converse relation, and the counit morphism $\epsilon_X : X \otimes X \to I$ is the relation with $(x_1, x_2) \epsilon_X *$ iff $x_1 = x_2$.
\end{proposition}

\begin{proof}
	Straightforward.
\end{proof}

Finally, a surprising pair of examples can be found in the categories of simple graphs and $\cup$-matrices \cite{chantawibul_sobocinski_towards_compositional_graph_theory}.
We illustrate the former.
Recall that a PROP is a strict symmetric monoidal category whose objects are natural numbers in which the tensor product of natural numbers is addition.
(A prototypical example is categories of matrices with the Kronecker product.)
A \emph{symmetric monoidal theory} is a PROP with a presentation by generators and relations.
It is shown in \cite{chantawibul_sobocinski_towards_compositional_graph_theory} that a certain PROP $\mathbf{CGraph}$ of \emph{simple graphs} is equivalent to a symmetric monoidal theory with generators
\[ \{ \Delta : 1 \to 2, \bot : 1 \to 0, \nabla : 2 \to 1, \top : 0 \to 1, \cup : 2 \to 0, v : 0 \to 1 \} \]
represented graphically as
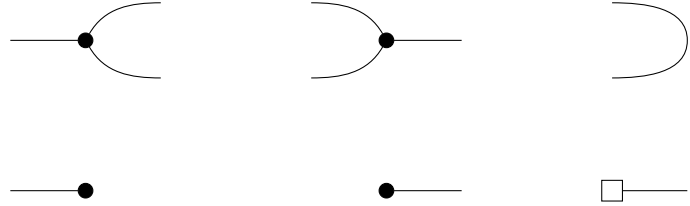
\begin{center} \begin{tikzpicture}
	\node [circle, scale=0.5, fill=black, draw] (A) at (0, 0) {}; \node [circle, scale=0.5, fill=black, draw] (B) at (0, -2) {};
	\node [circle, scale=0.5, fill=black, draw] (C) at (4, 0) {}; \node [circle, scale=0.5, fill=black, draw] (D) at (4, -2) {};
	\node [rectangle, draw] (E) at (7, -2) {};
	\draw [-] (-1, 0) to (A); \draw [-] (A) to [out=-60, in=180] (1, -.5); \draw [-] (A) to [out=60, in=180] (1, .5);
	\draw [-] (-1, -2) to (B);
	\draw [-] (3, -.5) to [out=0, in=-120] (C); \draw [-] (3, .5) to [out=0, in=120] (C); \draw [-] (C) to (5, 0);
	\draw [-] (D) to (5, -2);
	\draw [-] (7, -.5) to [out=0, in=-90] (8, 0) to [out=90, in=0] (7, .5);
	\draw [-] (E) to (8, -2);
\end{tikzpicture} \end{center}
together with a list of equations that includes
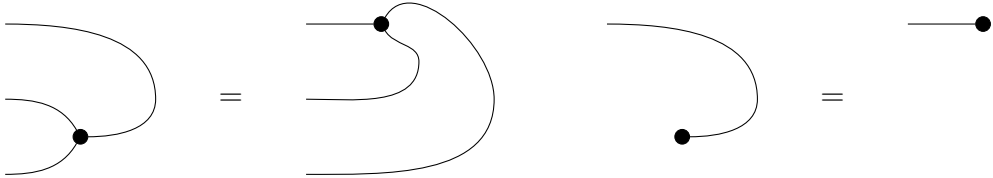
\begin{center} \begin{tikzpicture}
	\node [circle, scale=0.5, fill=black, draw] (A) at (0, .5) {};
	\node at (2, 1) {$=$};
	\node [circle, scale=0.5, fill=black, draw] (B) at (4, 2) {};
	\node [circle, scale=0.5, fill=black, draw] (C) at (8, .5) {};
	\node at (10, 1) {$=$};
	\node [circle, scale=0.5, fill=black, draw] (D) at (12, 2) {};
	\draw [-] (-1, 0) to [out=0, in=-120] (A); \draw [-] (-1, 1) to [out=0, in=120] (A);
	\draw [-] (A) to [out=0, in=-90] (1, 1) to [out=90, in=0] (-1, 2);
	\draw [-] (3, 0) to [out=0, in=-90] (5.5, 1) to [out=90, in=60] (B);
	\draw [-] (3, 1) to [out=0, in=-90] (4.5, 1.5) to [out=90, in=-60] (B);
	\draw [-] (3, 2) to (B);
	\draw [-] (C) to [out=0, in=-90] (9, 1) to [out=90, in=0] (7, 2);
	\draw [-] (11, 2) to (D);
\end{tikzpicture} \end{center}

We give $\mathbf{CGraph}$ the structure of a teleological category as follows.
Let $\mathbf{CGraph}_d$ be the sub-PROP of $\mathbf{CGraph}$ generated by $\{ \Delta, \bot, \nabla, \top \}$.
There is a PROP homomorphism $-^* : \mathbf{CGraph}_d^\op \to \mathbf{CGraph}$ generated by $\Delta^* = \nabla$, $\bot^* = \top$, $\nabla^* = \Delta$, $\top^* = \bot$.
(PROP homomorphisms are required to be identity-on-objects, so objects of $\mathbf{CGraph}$ are self-dual.)

We define the counit morphisms $\epsilon_n : n \oplus n \to 0$ recursively on $n$.
The base case is $\epsilon_0 = \id_0$, and the recursion is
\[ \epsilon_{n + 1} : n \oplus 1 \oplus n \oplus 1 \xrightarrow{\sigma_{n, 1} \oplus n \oplus 1} 1 \oplus n \oplus n \oplus 1 \xrightarrow{1 \oplus \epsilon_n \oplus 1} 1 \oplus 1 \xrightarrow{\cup} 0 \]

\begin{proposition}
	With these definitions, $\mathbf{CGraph}$ is a teleological category.
\end{proposition}

\begin{proof}
	It remains to prove extranaturality of $\epsilon$.
	This can easily be proven by structural induction on morphisms of $\mathbf{CGraph}_d$.
\end{proof}

\begin{definition}\label{def-teleological-functor}
	Let $\C$ and $\D$ be teleological categories.
	A \emph{teleological functor} $F : \C \to \D$ is a symmetric monoidal functor such that
	\begin{itemize}
		\item $F$ restricts to a symmetric monoidal functor $F_d : \C_d \to \D_d$ such that the diagram
		\begin{center} \begin{tikzpicture}[node distance=3cm, auto]
			\node (A) {$\C_d$}; \node (B) [right of=A] {$\D_d$}; \node (C) [below of=A] {$\C_d^\op$}; \node (D) [below of=B] {$\D_d^\op$};
			\draw [->] (A) to node {$F_d$} (B); \draw [->] (C) to node {$F_d^\op$} (D); \draw [->] (C) to node {$-^*$} (A); \draw [->] (D) to node {$-^*$} (B);
		\end{tikzpicture} \end{center}
		commutes
		
		\item $F (\epsilon_X) = \epsilon_{F (X)}$ for all objects $X$ of $\C$
	\end{itemize}
\end{definition}

\begin{lemma}
	Teleological categories and functors form a category $\TelCat$, with a forgetful functor $U : \TelCat \to \SymMonCat$.
\end{lemma}

Teleological natural transformations can also be defined in the obvious way, yielding a 2-category.

There is an identity-on-objects faithful teleological functor $\Lens \to \Game$, defined as follows.
A lens $\lambda : (X, S) \to (Y, R)$ is taken to a game, also denoted $\lambda$, with $\Sigma_\lambda = 1$ and $\mathbf B_\lambda (x, k) (*) = \{ * \}$ for all $x, k$.
(Games with these properties are called \emph{strategically trivial} and could also be called `zero-player open games'.)
The play and coplay functions are given by $\mathbf P_\lambda (*, x) = v_\lambda (x)$ and $\mathbf C_\lambda (*, x, r) = u_\lambda (x, r)$.
The image of adaptors under this functor are precisely the dualisable open games.
Since the functor is faithful, it identifies $\Lens$ as the wide `teleological subcategory' of $\Game$ consisting of strategically trivial open games.

\section{Teleological diagrams}\label{sec:teleological-diagrams}

\begin{definition}
	A \emph{teleological signature} $\Sigma$ consists of sets $\obj (\Sigma)$, $\mor (\Sigma)$ of \emph{object symbols} and \emph{morphism symbols}, together with a chosen subset $\mor_d (\Sigma) \subseteq \mor (\Sigma)$ of \emph{dualisable morphism symbols}, and for each morphism symbol $f$, a pair of words $\dom (f)$, $\cod (f)$ over $\obj (\Sigma) \cup \obj (\Sigma)^*$, where $\obj (\Sigma)^*$ is the set of formal symbols $x^*$ for object symbols $x$.
\end{definition}

If we ignore the set of dualisable morphism symbols (which is peculiar to teleological categories), the remainder of this definition resembles a specialisation of autonomous signatures to pivotal categories \cite[section 4]{selinger11}, which are the special (and common) case of autonomous categories in which left and right duals coincide, and are involutive.

For words we use the notation $x_1^{o_1} \otimes \cdots \otimes x_n^{o_n}$, where each $o_i$ is either $*$ or nothing.
We adopt the convention that $x^{**} = x$ is a syntactic identity, i.e. that $-^*$ denotes removing a star if there is already one.

\begin{definition}
	A \emph{teleological diagram} over a teleological signature $\Sigma$ is a diagram with the following properties.
	Edges are labelled by object symbols and are also oriented.
	Internal nodes are labelled by morphism symbols, where if a node $\alpha$ is labelled by a morphism symbol $f : x_1^{o_1} \otimes \cdots \otimes x_m^{o_m} \to y_1^{o'_1} \otimes \cdots \otimes y_n^{o'_n}$ then:
	\begin{itemize}
		\item If $f : \mor (\Sigma) \setminus \mor_d (\Sigma)$ (that is, $f$ is a non-dualisable morphism symbol) then $\alpha$ has the shape
		
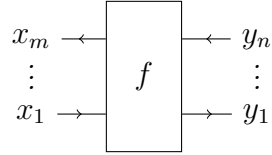
\begin{center} \begin{tikzpicture}
			\node [rectangle, minimum height=2cm, minimum width=1cm, draw] (f) at (0, 0) {$f$};
			\node (x1) at (-1.5, -.5) {$x_1$}; \node at (-1.5, .1) {$\vdots$}; \node (xm) at (-1.5, .5) {$x_m$};
			\node (y1) at (1.5, -.5) {$y_1$}; \node at (1.5, .1) {$\vdots$}; \node (yn) at (1.5, .5) {$y_n$};
			\draw [->-] (x1) to (f.west |- x1); \draw [-<-] (xm) to (f.west |- xm);
			\draw [-<-] (y1) to (f.east |- y1); \draw [->-] (yn) to (f.east |- yn);
		\end{tikzpicture} \end{center}
		where each arrow is oriented backwards (i.e. to the left) if the corresponding letter is decorated with $^*$, and forwards if it is not.
		We call $\alpha$ a \emph{non-dualisable node}.
		
		\item If $f : \mor_d (\Sigma)$ is a dualisable morphism symbol then $\alpha$ has one of the two shapes
		
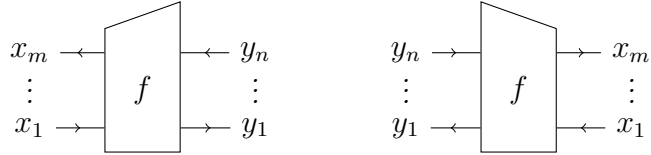
\begin{center} \begin{tikzpicture}
			\node [trapezium, trapezium left angle=0, trapezium right angle=75, shape border rotate=90, trapezium stretches=true, minimum height=1cm, minimum width=2cm, draw] (f) at (0, 0) {$f$};
			\node (x1) at (-1.5, -.5) {$x_1$}; \node at (-1.5, .1) {$\vdots$}; \node (xm) at (-1.5, .5) {$x_m$};
			\node (y1) at (1.5, -.5) {$y_1$}; \node at (1.5, .1) {$\vdots$}; \node (yn) at (1.5, .5) {$y_n$};
			\draw [->-] (x1) to (f.west |- x1); \draw [-<-] (xm) to (f.west |- xm);
			\draw [-<-] (y1) to (f.east |- y1); \draw [->-] (yn) to (f.east |- yn);
			\node [trapezium, trapezium left angle=75, trapezium right angle=0, shape border rotate=270, trapezium stretches=true, minimum height=1cm, minimum width=2cm, draw] (f') at (5, 0) {$f$};
			\node (x1') at (6.5, -.5) {$x_1$}; \node at (3.5, .1) {$\vdots$}; \node (xm') at (6.5, .5) {$x_m$};
			\node (y1') at (3.5, -.5) {$y_1$}; \node at (6.5, .1) {$\vdots$}; \node (yn') at (3.5, .5) {$y_n$};
			\draw [->-] (x1') to (f'.east |- x1'); \draw [-<-] (xm') to (f'.east |- xm');
			\draw [-<-] (y1') to (f'.west |- y1'); \draw [->-] (yn') to (f'.west |- yn');
		\end{tikzpicture} \end{center}
		differing by a reflection.
		We define the \emph{parity} $\pi (\alpha) = 0, 1$ to distinguish these cases.
	\end{itemize}
	
	Teleological diagrams must additionally satisfy the \emph{teleological condition}, which says that a string pointing backwards in the time-direction cannot bend around to point forwards.
	Each edge can be parameterised as $[0, 1] \to \R \times [a, b]$, which we write $t \mapsto (x (t), y (t))$.
	We require that any stationary point of $y$ is not a minimum.
	Since the parameterisation is smooth, we can impose the (slightly stronger) condition that if $y' (t) = 0$ then $y'' (t) \leq 0$.
\end{definition}

We extend the parity function to all nodes by defining $\pi (\alpha) = 0$ when $\alpha$ is a non-dualisable internal node or an external node.

We \emph{assume} that our diagrams are sufficiently well-behaved to make the following lemma true.
It is sufficient, for example, for directed edges to have continuously-varying tangent vectors, in which case the lemma follows from the intermediate value theorem.

\begin{lemma}\label{lemma-edge-bending}
	Let $A$ be a teleological diagram, and $e$ be a directed edge in $A$ smoothly parameterised by $[0, 1] \to \R \times [a, b]$, $t \mapsto (x (t), y (t))$.
	Then one of the following is the case:
	\begin{itemize}
		\item $y$ is strictly increasing
		\item $y$ is strictly decreasing
		\item There is a point $0 < t < 1$ such that $y$ is strictly increasing on $[0, t)$, has a maximum at $t$, and is strictly decreasing on $(t, 1]$, and there is a neighbourhood of $t$ on which $x$ is either strictly increasing or strictly decreasing, so that in this neighbourhood the edge has one of the two forms
		
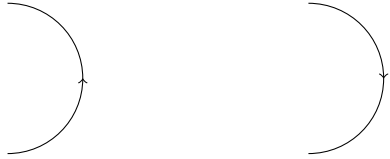
\begin{center} \begin{tikzpicture}
			\draw [->-] (0, 0) to [out=0, in=-90] (1, 1) to [out=90, in=0] (0, 2);
			\draw [->-] (4, 2) to [out=0, in=90] (5, 1) to [out=-90, in=0] (4, 0);
		\end{tikzpicture} \end{center}
	\end{itemize}
\end{lemma}

This is in the spirit of similar simplifying assumptions made in the literature to prove coherence theorems (see \cite{selinger11}), but contradicts the assumption in \cite{joyal_street_planar_diagrams_tensor_algebra} for autonomous categories that diagrams are piecewise linear.
The previous lemma could also be proved by assuming that diagrams are piecewise linear and that no piece is parallel to the $x$-axis.
(It is possible to motivate this by thinking of $|\mathrm{d}x / \mathrm{d}t|$ as a notion of `signalling speed' that may not become infinity.)

\begin{lemma}\label{lem:progressive-implies-teleological}
	The progressive condition on diagrams implies the teleological condition.
\end{lemma}

\begin{proof}
	Suppose the teleological condition fails for some edge $e$ with parameterisation $t \mapsto (x (t), y (t))$, so $y (t)$ is a local maximum.
	Then for sufficiently small $\epsilon > 0$, the edge $e$ intersects the time-slice $\R \times \{ y (t) - \epsilon \}$ in two places.
\end{proof}

As for circuit diagrams, internal nodes determine an ordering on the adjacent edges, and diagrams determine an ordering on the edges adjacent to external nodes.
We write this as
\[ \alpha, A : e_1^{o_1} \otimes \cdots \otimes e_m^{o_m} \to e'_1{}^{o'_1} \otimes \cdots \otimes e'_n{}^{o'_n} \]
where the $o_i, o'_j$ are either $*$ if the edge is oriented backwards in the time direction (that is, if $y$ is decreasing in a neighbourhood of the node), or nothing if it is oriented forwards.

\begin{definition}
	Let $A$ and $B$ be teleological diagrams in $\R \times [a, b]$ over a teleological signature $\Sigma$.
	A \emph{isomorphism} $i : A \cong_t B$ between $A$ and $B$ consists of a label-preserving bijection between the nodes of $A$ and $B$, and a label-preserving bijection between the edges of $A$ and $B$, both written $i$, such that
	\begin{itemize}
		\item For each internal node $\alpha$ of $A$ with $\pi (\alpha) = \pi (i (\alpha))$,
		\[ \alpha : e_1^{o_1} \otimes \cdots \otimes e_m^{o_1} \to e'_1{}^{o'_1} \otimes \cdots \otimes e'_n{}^{o'_n} \]
		iff
		\[ i (\alpha) : i (e_1)^{o_1} \otimes \cdots \otimes i (e_m)^{o_m} \to i (e'_1)^{o'_1} \otimes \cdots \otimes i (e'_n)^{o'_n} \]
		
		\item For each internal node $\alpha$ of $A$ with $\pi (\alpha) \neq \pi (i (\alpha))$,
		\[ \alpha : e_1^{o_1} \otimes \cdots \otimes e_m^{o_m} \to e'_1{}^{o'_1} \otimes \cdots \otimes e'_n{}^{o'_n} \]
		iff
		\[ i (\alpha) : i (e'_1)^{o'_1 *} \otimes \cdots \otimes i (e'_n)^{o'_n *} \to i (e_1)^{o_1 *} \otimes \cdots \otimes i (e_m)^{o_m *} \]
		
		\item For each external node $\beta$ of $A$, $\beta$ is adjacent to edge $e$ iff $i (\beta)$ is adjacent to edge $i (e)$, and
		\[ A : e_1^{o_1} \otimes \cdots \otimes e_m^{o_m} \to e'_1{}^{o'_1} \otimes \cdots \otimes e'_n{}^{o'_n} \]
		iff
		\[ B : i (e_1)^{o_1} \otimes \cdots \otimes i (e_m)^{o_m} \to i (e'_1)^{o'_1} \otimes \cdots \otimes i (e'_n)^{o'_n} \]
	\end{itemize}
\end{definition}

In words, an isomorphism $A \cong_t B$ is an oriented version of equivalence of string diagrams (definition \ref{def:circuit-diagram-isomorphism}), in which only dualisable nodes may optionally be reflected.

\begin{lemma}
	$\cong_t$ is an equivalence relation on teleological diagrams.
\end{lemma}

\begin{proof}
	Straightforward.
\end{proof}

By the usual method of \cite{joyal91} we build a symmetric monoidal category $\D_t (\Sigma)$ whose objects are words over $\Sigma \cup \Sigma^*$, and whose morphisms are $\cong_t$-equivalence classes of teleological categories.

We take the category $\D_t (\Sigma)_d$ of dualisable morphisms to be the wide subcategory consisting of teleological diagrams whose internal nodes are all labelled by dualisable morphism symbols, and which satisfy the progressive condition.
The functor $i : \D_t (\Sigma)_d \to \D_t (\Sigma)$ is inclusion.
This is a subcategory by lemma \ref{lem:progressive-implies-teleological}, and because the identity morphism on a word is a diagram containing no internal nodes and which is progressive.

Duality on words is defined by $(x_1^{o_1} \otimes \cdots \otimes x_n^{o_n})^* = x_1^{o_1 *} \otimes \cdots \otimes x_n^{o_n *}$, recalling our convention that $x^{**} = x$.
Duality on diagrams whose nodes are all dualisable is defined by reflection in the horizontal (time) direction.
For example, the dual of the diagram
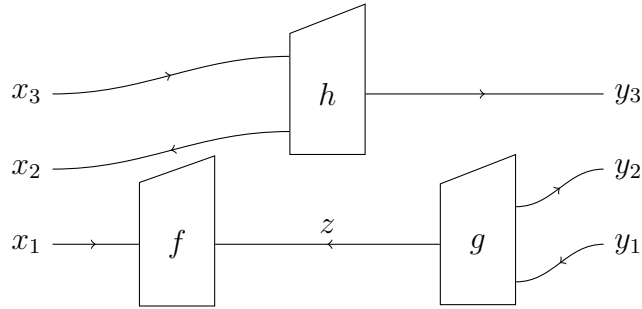
\begin{center} \begin{tikzpicture}
	\node (x1) at (0, 0) {$x_1$}; \node (x2) at (0, 1) {$x_2$}; \node (x3) at (0, 2) {$x_3$};
	\node (y1) at (8, 0) {$y_1$}; \node (y2) at (8, 1) {$y_2$}; \node (y3) at (8, 2) {$y_3$};
	\node [trapezium, trapezium left angle=0, trapezium right angle=75, shape border rotate=90, trapezium stretches=true, minimum height=1cm, minimum width=2cm, draw] (f) at (2, 0) {$f$};
	\node [trapezium, trapezium left angle=0, trapezium right angle=75, shape border rotate=90, trapezium stretches=true, minimum height=1cm, minimum width=2cm, draw] (g) at (6, 0) {$g$};
	\node [trapezium, trapezium left angle=0, trapezium right angle=75, shape border rotate=90, trapezium stretches=true, minimum height=1cm, minimum width=2cm, draw] (h) at (4, 2) {$h$};
	\node (d1) at (0, 1.5) {}; \node (d2) at (0, 2.5) {}; \node (d3) at (8, -.5) {}; \node (d4) at (8, .5) {};
	\draw [->-] (x1) to [out=0, in=180] (f); \draw [->-] (h.west |- d1) to [out=180, in=0] (x2); \draw [->-] (x3) to [out=0, in=180] (h.west |- d2);
	\draw [->-] (y1) to [out=180, in=0] (g.east |- d3); \draw [->-] (g.east |- d4) to [out=0, in=180] (y2); \draw [->-] (h) to (y3);
	\draw [->-] (g) to node [above] {$z$} (f);
\end{tikzpicture} \end{center}
(which is a morphism $x_1 \otimes x_2^* \otimes x_3 \to y_1^* \otimes y_2 \otimes y_3$ in $\D_t (\Sigma)_d$) is
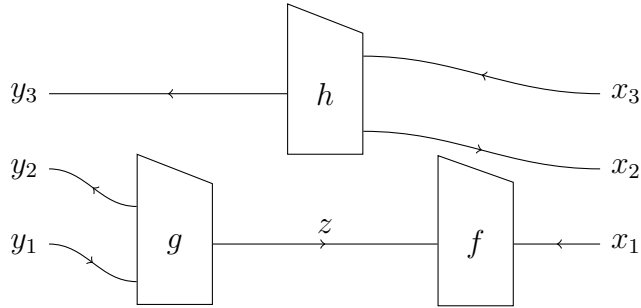
\begin{center} \begin{tikzpicture}
	\node (x1) at (8, 0) {$x_1$}; \node (x2) at (8, 1) {$x_2$}; \node (x3) at (8, 2) {$x_3$};
	\node (y1) at (0, 0) {$y_1$}; \node (y2) at (0, 1) {$y_2$}; \node (y3) at (0, 2) {$y_3$};
	\node [trapezium, trapezium left angle=75, trapezium right angle=0, shape border rotate=270, trapezium stretches=true, minimum height=1cm, minimum width=2cm, draw] (f) at (6, 0) {$f$};
	\node [trapezium, trapezium left angle=75, trapezium right angle=0, shape border rotate=270, trapezium stretches=true, minimum height=1cm, minimum width=2cm, draw] (g) at (2, 0) {$g$};
	\node [trapezium, trapezium left angle=75, trapezium right angle=0, shape border rotate=270, trapezium stretches=true, minimum height=1cm, minimum width=2cm, draw] (h) at (4, 2) {$h$};
	\node (d1) at (0, 1.5) {}; \node (d2) at (0, 2.5) {}; \node (d3) at (8, -.5) {}; \node (d4) at (8, .5) {};
	\draw [->-] (x1) to [out=180, in=0] (f); \draw [->-] (h.east |- d1) to [out=0, in=180] (x2); \draw [->-] (x3) to [out=180, in=0] (h.east |- d2);
	\draw [->-] (y1) to [out=0, in=180] (g.west |- d3); \draw [->-] (g.west |- d4) to [out=180, in=0] (y2); \draw [->-] (h) to (y3);
	\draw [->-] (g) to node [above] {$z$} (f);
\end{tikzpicture} \end{center}
(which is a morphism $y_1 \otimes y_2^* \otimes y_3^* \to x_1^* \otimes x_2 \otimes x_3^*$).
It is straightforward to show that this construction is well-defined on $\cong_t$-equivalence classes, and that it defines a symmetric monoidal functor $-^* : \D_t (\Sigma)_d^\op \to \D_t (\Sigma)_d$.

Given an object $x_1^{o_1} \otimes \cdots \otimes x_n^{o_n}$, the counit
\[ \epsilon_{x_1^{o_1} \otimes \cdots \otimes x_n^{o_n}} : x_1^{o_1} \otimes \cdots \otimes x_n^{o_n} \otimes x_1^{o_1 *} \otimes \cdots \otimes x_n^{o_n *} \to I \]
in $\D_t (\Sigma)$ is (the $\cong_t$-equivalence class of) the obvious diagram consisting of $n$ interleaved caps.
For example,
\[ \epsilon_{x_1 \otimes x_2^* \otimes x_3} : x_1 \otimes x_2^* \otimes x_3 \otimes x_1^* \otimes x_2 \otimes x_3^* \to I \]
is the diagram depicted in figure \ref{fig:counit-example}.
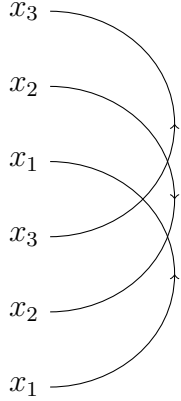
\begin{figure}
\begin{center} \begin{tikzpicture}
	\node (x1) at (0, 0) {$x_1$}; \node (x2) at (0, 1) {$x_2$}; \node (x3) at (0, 2) {$x_3$};
	\node (x1') at (0, 3) {$x_1$}; \node (x2') at (0, 4) {$x_2$}; \node (x3') at (0, 5) {$x_3$};
	\draw [->-] (x1) to [out=0, in=-90] (2, 1.5) to [out=90, in=0] (x1');
	\draw [->-] (x2') to [out=0, in=90] (2, 2.5) to [out=-90, in=0] (x2);
	\draw [->-] (x3) to [out=0, in=-90] (2, 3.5) to [out=90, in=0] (x3');
\end{tikzpicture} \end{center}
\caption{Example of counit in $\D_t (\Sigma)$}
\label{fig:counit-example}
\end{figure}
Extranaturality requires certain equalities between morphisms in $\D_t (\Sigma)$, which means isomorphisms between diagrams.
These can easily be seen.
For example, for a dualisable morphism symbol $f : x_1 \otimes x_2^* \to y_1^* \otimes y_2$ there is an isomorphism
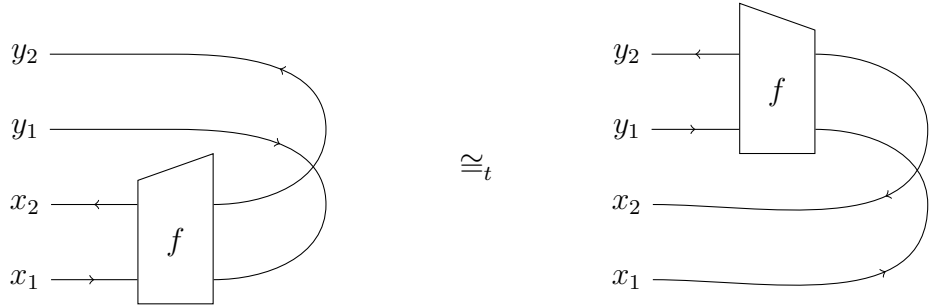
\begin{center} \begin{tikzpicture}
	\node (x1) at (0, 0) {$x_1$}; \node (x2) at (0, 1) {$x_2$}; \node (y1) at (0, 2) {$y_1$}; \node (y2) at (0, 3) {$y_2$};
	\node [trapezium, trapezium left angle=0, trapezium right angle=75, shape border rotate=90, trapezium stretches=true, minimum height=1cm, minimum width=2cm, draw] (f) at (2, .5) {$f$};
	\draw [->-] (x1) to (f.west |- x1); \draw [->-] (f.west |- x2) to (x2);
	\draw [->-] (y1) to (2, 2) to [out=0, in=90] (4, 1) to [out=-90, in=0] (f.east |- x1); \draw [->-] (f.east |- x2) to [out=0, in=-90] (4, 2) to [out=90, in=0] (2, 3) to (y2);
	\node at (6, 1.5) {$\cong_t$};
	\node (x1') at (8, 0) {$x_1$}; \node (x2') at (8, 1) {$x_2$}; \node (y1') at (8, 2) {$y_1$}; \node (y2') at (8, 3) {$y_2$};
	\node [trapezium, trapezium left angle=75, trapezium right angle=0, shape border rotate=270, trapezium stretches=true, minimum height=1cm, minimum width=2cm, draw] (f') at (10, 2.5) {$f$};
	\draw [->-] (x1') to [out=0, in=-90] (12, 1) to [out=90, in=0] (f'.east |- y1'); \draw [->-] (f'.east |- y2') to [out=0, in=90] (12, 2) to [out=-90, in=0] (x2');
	\draw [->-] (y1') to (f'.west |- y1'); \draw [->-] (f'.west |- y2') to (y2');
\end{tikzpicture} \end{center}

\begin{proposition}
	Let $\Sigma$ be a teleological signature.
	The above construction defines a teleological category $\D_t (\Sigma)$.
\end{proposition}

$\D_t (\Sigma)$ is also equipped with an obvious valuation $v_\Sigma : \Sigma \to \D_t (\Sigma)$, taking each morphism symbol $f$ to the diagram containing only one $f$-labelled internal node.

\section{Coherence for teleological diagrams}\label{sec:coherence}

In this section we will prove the following coherence theorem, which states that $\D_t (\Sigma)$ is the free teleological category on $\Sigma$.

\begin{theorem}[Coherence theorem for teleological categories]\label{main-theorem}
	Let $\Sigma$ be a teleological signature, $\C$ a teleological category and $w : \Sigma \to \C$ a valuation.
	Then there exists a teleological functor $F : \D_t (\Sigma) \to \C$ such that $w = F \circ v_\Sigma$, unique up to unique teleological natural isomorphism.
\end{theorem}

The following is a summary of the proof method.
Given an isomorphism of teleological diagrams, we show that it can be factored into three parts (lemma \ref{normal-form-lemma}): first translate nodes without reflecting, then reflect a single sub-diagram, then finally translate without reflecting again.
This is proved by a combinatorial analysis of the possible connections in a diagram, ultimately resulting on lemma \ref{lemma-edge-bending}.
For the isomorphisms not involving reflections we appeal to the Joyal-Street coherence theorem (theorem \ref{joyal-street-theorem}) for an expanding signature, and the single reflection is an instance of extranaturality.

We begin by putting lemma \ref{lemma-edge-bending} into a discrete form that is more practical to work with.
Consider a general internal node of a teleological diagram, either dualisable or non-dualisable.
There are four ways in which an edge can connect to the node:
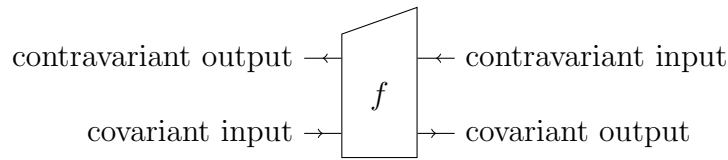
\begin{center} \begin{tikzpicture}
	\node [trapezium, trapezium left angle=0, trapezium right angle=75, shape border rotate=90, trapezium stretches=true, minimum height=1cm, minimum width=2cm, draw] (a) at (0, 0) {$f$};
	\node [anchor=east] (X) at (-1, -.5) {covariant input}; \node [anchor=west] (Y) at (1, -.5) {covariant output};
	\node [anchor=west] (R) at (1, .5) {contravariant input}; \node [anchor=east] (S) at (-1, .5) {contravariant output};
	\draw [->-] (X) to (a.west |- X); \draw [->-] (a.east |- Y) to (Y);
	\draw [->-] (R) to (a.east |- R); \draw [->-] (a.west |- S) to (S);
\end{tikzpicture} \end{center}
Notice that reflecting a dualisable node preserves inputs and outputs, but interchanges covariant and contravariant connections.

\begin{definition}
	Given a directed edge $e$ from an internal node $\alpha$ to an internal node $\beta$, we write $e : \alpha^p \to \beta^q$ where $p, q : \{ +, - \}$, where $+$ represents a covariant connection and $-$ a contravariant connection.
	We extend this notation to external nodes as follows.
	If $\beta$ is an external node in $\R \times \{ a \}$ then we write $\beta^+$ when it is the source node of its (unique) adjacent edge, and $\beta^-$ when it is the target node.
	If $\beta$ is in $\R \times \{ b \}$ then we write $\beta^-$ when it is the source node and $\beta^+$ when it is the target node.
\end{definition}

For example, if the connection between $e$ and $\alpha$ is a covariant output, and the connection between $e$ and $\beta$ is a contravariant input, we write $e : \alpha^+ \to \beta^-$.

\begin{definition}
	Given nodes $\alpha, \beta$ in a teleological diagram we write $\alpha < \beta$ if $\alpha$ strictly precedes $\beta$ in the time direction.
\end{definition}

This defines a preorder on nodes, in which two nodes are incomparable iff they are not equal and are located in the same time-slice.
Of course, this notion is not invariant under equivalences of diagrams.

\begin{lemma}\label{lem-forbidden-variance}
	Let $A$ be a teleological diagram, with a directed edge $e : \alpha \to \beta$.
	\begin{itemize}
		\item If $e : \alpha^+ \to \beta^+$ then $\alpha < \beta$.
		\item If $e : \alpha^- \to \beta^-$ then $\alpha > \beta$.
		\item It is not the case that $e : \alpha^- \to \beta^+$.
	\end{itemize}
\end{lemma}

\begin{proof}
	Follows from lemma \ref{lemma-edge-bending}.
\end{proof}

\begin{definition}
	Given an equivalence $i : A \cong_t B$ and a node $\alpha$ in $A$, if $\pi (\alpha) \neq \pi (i (\alpha))$ we write $i \reflects \alpha$ and say that $i$ \emph{reflects} $\alpha$; otherwise, we write $i \not\reflects \alpha$.
\end{definition}

Note that $i \not\reflects \alpha$ whenever $\alpha$ is a non-dualisable internal node or an external node.

\begin{lemma}\label{lem-reflection-variances}
	Let $A$ and $B$ be teleological diagrams and $i : A \cong_t B$, and let $e : \alpha \to \beta$ be an edge in $A$.
	If $i \not\reflects \alpha$ and $i \reflects \beta$ then one of the following is the case:
	\begin{itemize}
		\item $e : \alpha^+ \to \beta^+$, $i (e) : i (\alpha)^+ \to i (\beta)^-$ and $\alpha < \beta$
		\item $e : \alpha^+ \to \beta^-$, $i (e) : i (\alpha)^+ \to i (\beta)^+$ and $i (\alpha) < i (\beta)$
	\end{itemize}
	If $i \reflects \alpha$ and $i \not\reflects \beta$ then one of the following is the case:
	\begin{itemize}
		\item $e : \alpha^- \to \beta^-$, $i (e) : i (\alpha)^+ \to i (\beta)^-$ and $\alpha > \beta$
		\item $e : \alpha^+ \to \beta^-$, $i (e) : i (\alpha)^- \to i (\beta)^-$ and $i (\alpha) > i (\beta)$
	\end{itemize}
	If $i \reflects \alpha$ and $i \reflects \beta$ then one of the following is the case:
	\begin{itemize}
		\item $e : \alpha^+ \to \beta^+$, $i (e) : i (\alpha)^- \to i (\beta)^-$, $\alpha < \beta$ and $i (\alpha) > i (\beta)$
		\item $e : \alpha^- \to \beta^-$, $i (e) : i (\alpha)^+ \to i (\beta)^+$, $\alpha > \beta$ and $i (\alpha) < i (\beta)$
	\end{itemize}
\end{lemma}

\begin{proof}
	For each case we consider the 3 possible variances of $e : \alpha \to \beta$, with $e : \alpha^- \to \beta^+$ ruled out by lemma \ref{lem-forbidden-variance}.
	In each case the remaining variance not listed would imply that $i (e) : i (\alpha)^- \to i (\beta)^+$, which is again disallowed.
	The ordering constraints on the permitted cases also follow from lemma \ref{lem-forbidden-variance}.
\end{proof}

Examples of these three cases are illustrated in figure \ref{fig:reflection-variances}.

\begin{figure}
	\begin{center} \begin{tikzpicture}
		\node [rectangle, minimum height=2cm, minimum width=1cm, draw] (f1) at (0, 0) {$f$};
		\node [trapezium, trapezium left angle=0, trapezium right angle=75, shape border rotate=90, trapezium stretches=true, minimum height=1cm, minimum width=2cm, draw] (g1) at (2, 0) {$g$};
		\draw [->-] (f1) to node [above] {$x$} (g1);
		\node at (4, 0) {$\cong_t$};
		\node [rectangle, minimum height=2cm, minimum width=1cm, draw] (f2) at (6, -1) {$f$};
		\node [trapezium, trapezium left angle=75, trapezium right angle=0, shape border rotate=270, trapezium stretches=true, minimum height=1cm, minimum width=2cm, draw] (g2) at (6, 1) {$g$};
		\draw [->-] (f2) to [out=0, in=-90] (7.5, 0) to [out=90, in=0] node [right, very near start] {$x$} (g2);
		\node [rectangle, minimum height=2cm, minimum width=1cm, draw] (g3) at (0, -5) {$g$};
		\node [trapezium, trapezium left angle=0, trapezium right angle=75, shape border rotate=90, trapezium stretches=true, minimum height=1cm, minimum width=2cm, draw] (f3) at (2, -5) {$f$};
		\draw [->-] (f3) to node [above] {$x$} (g3);
		\node at (4, -5) {$\cong_t$};
		\node [rectangle, minimum height=2cm, minimum width=1cm, draw] (g4) at (6, -6) {$g$};
		\node [trapezium, trapezium left angle=75, trapezium right angle=0, shape border rotate=270, trapezium stretches=true, minimum height=1cm, minimum width=2cm, draw] (f4) at (6, -4) {$f$};
		\draw [->-] (f4) to [out=0, in=90] node [right, very near end] {$x$} (7.5, -5) to [out=-90, in=0] (g4);
		\node [trapezium, trapezium left angle=0, trapezium right angle=75, shape border rotate=90, trapezium stretches=true, minimum height=1cm, minimum width=2cm, draw] (f5) at (0, -9) {$f$};
		\node [trapezium, trapezium left angle=75, trapezium right angle=0, shape border rotate=270, trapezium stretches=true, minimum height=1cm, minimum width=2cm, draw] (g5) at (2, -9) {$g$};
		\draw [->-] (f5) to node [above] {$x$} (g5);
		\node at (4, -9) {$\cong_t$};
		\node [trapezium, trapezium left angle=75, trapezium right angle=0, shape border rotate=270, trapezium stretches=true, minimum height=1cm, minimum width=2cm, draw] (f6) at (8, -9) {$f$};
		\node [trapezium, trapezium left angle=0, trapezium right angle=75, shape border rotate=90, trapezium stretches=true, minimum height=1cm, minimum width=2cm, draw] (g6) at (6, -9) {$g$};
		\draw [->-] (f6) to node [above] {$x$} (g6);
	\end{tikzpicture} \end{center}
	\caption{Cases of lemma \ref{lem-reflection-variances}}
	\label{fig:reflection-variances}
\end{figure}

\begin{definition}\label{def:teleological-circuit-isomorphism}
	Let $\Sigma$ be a teleological signature and $A$, $B$ teleological diagrams over $\Sigma$.
	We write $i : A \cong_c B$ if $i : A \cong_t B$ and $i \not\reflects \alpha$ for all nodes $\alpha$ of $A$.
\end{definition}

\begin{lemma}[Normal-form lemma]\label{normal-form-lemma}
	Let $\Sigma$ be a teleological signature, and let $A$ and $B$ be teleological diagrams over $\Sigma$ with $A \cong_t B$.
	Then there exist teleological diagrams $A' \cong_c A$ and $B' \cong_c B$ such that $A'$ and $B'$ are respectively of the form depicted in figure \ref{fig:normal-form-lemma}, 	where $P$ and $Q$ are sub-diagrams with $Q$ in $\D_t (\Sigma)_d$.
	Note that the particular edge orientations depicted are for illustration, and that the reflected $Q$ in the diagram for $B'$ refers to the dual $Q^*$, i.e. the diagram obtained by reflecting the diagram $Q$.
\end{lemma}

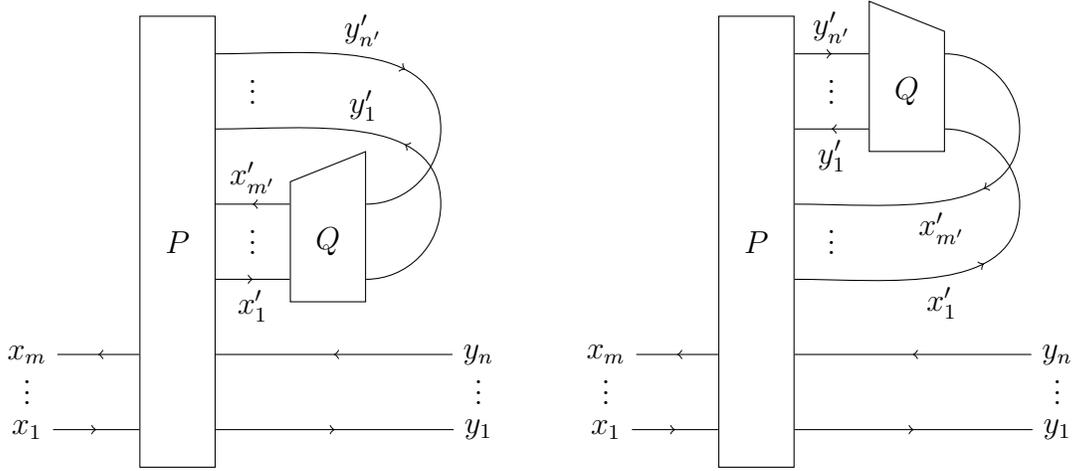
\begin{figure}
	\begin{center}
		\begin{tikzpicture}
			\node [rectangle, minimum height=6cm, minimum width=1cm, draw] (P) at (0, 0) {$P$};
			\node [trapezium, trapezium left angle=0, trapezium right angle=75, shape border rotate=90, trapezium stretches=true, minimum height=1cm, minimum width=2cm, draw] (Q) at (2, 0) {$Q$};
			\node (X) at (-2, -2.5) {$x_1$}; \node (S) at (-2, -1.5) {$x_m$};
			\node (Y) at (4, -2.5) {$y_1$}; \node (R) at (4, -1.5) {$y_n$};
			\node at (-2, -1.9) {$\vdots$}; \node at (4, -1.9) {$\vdots$};
			\node at (1, 0.1) {$\vdots$}; \node at (1, 2.1) {$\vdots$};
			\node (d1) at (0, -.5) {}; \node (d2) at (0, .5) {}; \node (d3) at (0, 1.5) {}; \node (d4) at (0, 2.5) {};
			\draw [->-] (X) to (P.west |- X); \draw [->-] (P.west |- S) to (S);
			\draw [->-] (P.east |- Y) to (Y); \draw [->-] (R) to (P.east |- R);
			\draw [->-] (P.east |- d1) to node [below] {$x'_1$} (Q.west |- d1); \draw [->-] (Q.west |- d2) to node [above] {$x'_{m'}$} (P.east |- d2);
			\draw [->-] (Q.east |- d1) to [out=0, in=-90] (3.5, .5) to [out=90, in=0] node [above] {$y'_1$} (P.east |- d3);
			\draw [->-] (P.east |- d4) to [out=0, in=90] node [above] {$y'_{n'}$} (3.5, 1.5) to [out=-90, in=0] (Q.east |- d2);
		\end{tikzpicture}
		\qquad
		\begin{tikzpicture}
			\node [rectangle, minimum height=6cm, minimum width=1cm, draw] (P) at (0, 0) {$P$};
			\node [trapezium, trapezium left angle=75, trapezium right angle=0, shape border rotate=270, trapezium stretches=true, minimum height=1cm, minimum width=2cm, draw] (Q) at (2, 2) {$Q$};
			\node (X) at (-2, -2.5) {$x_1$}; \node (S) at (-2, -1.5) {$x_m$};
			\node (Y) at (4, -2.5) {$y_1$}; \node (R) at (4, -1.5) {$y_n$};
			\node at (-2, -1.9) {$\vdots$}; \node at (4, -1.9) {$\vdots$};
			\node at (1, 0.1) {$\vdots$}; \node at (1, 2.1) {$\vdots$};
			\node (d1) at (0, -.5) {}; \node (d2) at (0, .5) {}; \node (d3) at (0, 1.5) {}; \node (d4) at (0, 2.5) {};
			\draw [->-] (X) to (P.west |- X); \draw [->-] (P.west |- S) to (S);
			\draw [->-] (P.east |- Y) to (Y); \draw [->-] (R) to (P.east |- R);
			\draw [->-] (P.east |- d1) to [out=0, in=-90] node [below] {$x'_1$} (3.5, .5) to [out=90, in=0] (Q.east |- d3);
			\draw [->-] (Q.east |- d4) to [out=0, in=90] (3.5, 1.5) to [out=-90, in=0] node [below] {$x'_{m'}$} (P.east |- d2);
			\draw [->-] (Q.west |- d3) to node [below] {$y'_1$} (P.east |- d3); \draw [->-] (P.east |- d4) to node [above] {$y'_{n'}$} (Q.west |- d4);
		\end{tikzpicture}
	\end{center}
	\caption{$A'$ and $B'$ in lemma \ref{normal-form-lemma}}
	\label{fig:normal-form-lemma}
\end{figure}

\begin{proof}
	Let $i : A \cong_t B$ be the equivalence.
	We will construct $P$ and $Q$ such that $P$ contains the nodes $\alpha$ of $A$ with $i \not\reflects \alpha$, and $Q$ contains those with $i \reflects \alpha$.
	The external nodes of $A$ are of the former kind, and add them as external nodes to $P$.
	(These are the $x_i$, $y_j$ in figure \ref{fig:normal-form-lemma}.)
	We preserve the relative positions of the nodes within these sub-diagrams.
	We must check all of the ways in which nodes may be connected.
	
	Any edge $e : \alpha \to \beta$ in $A$ between nodes of $P$ is also added to the diagram $P$.
	Similarly we add any edge in $A$ between nodes of $Q$ to $Q$.
	In the latter case, by lemma \ref{lem-reflection-variances} it must be the case that either $e : \alpha^+ \to \beta^+$ or $e : \alpha^- \to \beta^-$.
	
	For each edge $e : \alpha \to \beta$ for $\alpha$ in $P$ and $\beta$ in $Q$, by lemma \ref{lem-reflection-variances} either $e : \alpha^+ \to \beta^+$ or $e : \alpha^+ \to \beta^-$.
	In both cases we add an edge $e_p : \alpha^+ \to p (e)^+$ to $P$, for $p (e)$ a fresh covariant output external node of $P$.
	In the former case we also add an edge $e_q : q (e)^+ \to \beta^+$ to $Q$, for $q (e)$ a fresh covariant input external node of $Q$, and an edge $p (e)^+ \to q (e)^+$ to $A'$ (of the form $x'_1$ in figure \ref{fig:normal-form-lemma}).
	In the latter case we add $e_q : q (e)^- \to \beta^-$, for $q (e)$ a fresh contravariant input external node of $Q$, and an edge $p (e)^+ \to q (e)^-$ to $A'$ (of the form $y'_{n'}$).
	
	In the dual case $e : \alpha \to \beta$ for $\alpha$ in $Q$ and $\beta$ in $P$, by lemma \ref{lem-reflection-variances} either $e : \alpha^- \to \beta^-$ or $e : \alpha^+ \to \beta^-$.
	In either case we add an edge $e_p : p(e)^- \to \beta^-$, for $p (e)$ a fresh contravariant input external node of $P$.
	In the former case we also add an edge $e_q : \alpha^- \to q (e)^-$ to $Q$, for $q (e)$ a fresh contravariant output external node of $Q$, and an edge $q (e)^- \to p (e)^-$ to $A'$ (of the form $x'_{m'}$).
	In the latter case we add $e_q : \alpha^+ \to q (e)^+$, for $q (e)$ a fresh covariant output external node of $Q$, and an edge $q (e)^+ \to p (e)^-$ to $A'$ (of the form $y'_1$).
	
	Considering external nodes of $A$ as external nodes of $P$ means that an edge $e : \alpha \to \beta$ for $\alpha$ a covariant input external node of $A$ (of the form $x_1$) and $\beta$ in $Q$ leads to a `bypass' edge $e_p : \alpha^+ \to e (p)^+$ in $P$ that does not connect to any internal node of $P$; similarly for $e : \alpha \to \beta^-$ for $\alpha$ in $Q$ and $\beta$ a contravariant output external node of $A$ (of the form $x_m$).
	Furthermore, any edge (in either direction) between a node of $Q$ and either a covariant output or contravariant input external node of $A$ would necessarily be either of the form $e : \alpha^+ \to \beta^+$ with $i \reflects \alpha$ and $i \not\reflects \beta$, or $e : \alpha^- \to \beta^-$ with $i \not\reflects \alpha$ and $i \reflects \beta$, both of which are ruled out by lemma \ref{lem-reflection-variances}.
	
	All edges added to $Q$ by this construction are of the form $\alpha^+ \to \beta^+$ or $\alpha^- \to \beta^-$.
	It follows that $Q$ satisfies the progressive condition.
	
	We have $A \cong_c A'$ by construction, and it is also straightforward to see that $B \cong_c B'$.
\end{proof}

\begin{definition}
	Let $\Sigma$ be a teleological signature.
	We define a monoidal signature $M (\Sigma)$ as follows.
	The object symbols are
	\[ \obj (M (\Sigma)) = \obj (\Sigma) \cup \obj (\Sigma)^* = \obj (\Sigma) \cup \{ x^* \mid x : \obj (\Sigma) \} \]
	and the morphism symbols are
	\[ \mor (M (\Sigma)) = \mor (\Sigma) \cup \mor_d (\Sigma)^* \cup \{ \epsilon_x \mid x : \obj (\Sigma) \} \]
	If $f : x_1^{o_1} \otimes \cdots \otimes x_m^{o_m} \to y_1^{o'_1} \otimes \cdots \otimes y_n^{o'_n}$ in $\Sigma$ then it also has this type in $M (\Sigma)$, together with
	\[ f^* : y_1^{o'_1 *} \otimes \cdots \otimes y_n^{o'_n *} \to x_1^{o_1 *} \otimes \cdots \otimes x_m^{o_m *} \]
	with $x^{**} = x$ as a syntactic equality.
	Finally, we set $\epsilon_x : x \otimes x^* \to I$ for each $x : \obj (\Sigma)$.
\end{definition}

\begin{lemma}
	Let $\Sigma$ be a teleological signature.
	There is an identity-on-objects symmetric monoidal functor $E : \D_c (M (\Sigma)) \to U (\D_t (\Sigma))$ defined by embedding $\cong_c$-equivalence classes into $\cong_t$-equivalence classes, where $U$ is the forgetful functor $\TelCat \to \SymMonCat$.
\end{lemma}

\begin{definition}
	Let $\Sigma$ be a teleological signature and $A$ a teleological diagram over $\Sigma$.
	We define a circuit diagram $C (A)$ over $M (\Sigma)$ as follows.
	For each node $\alpha$ of $A$ labelled by a morphism symbol $f$ we add a node at the same position in $C (A)$ labelled by $f$ if $\pi (\alpha) = 0$, and labelled by $f^*$ if $\pi (\alpha) = 1$.
	For each edge $e$ labelled by an object symbol $x$, let $t \mapsto (x (t), y (t))$ be a smooth parameterisation, and case split on lemma \ref{lemma-edge-bending}.
	\begin{itemize}
		\item If $y$ is strictly increasing, we add an edge in $C (A)$ with the same parameterisation, labelled by $x$.
		\item If $y$ is strictly decreasing, we add an edge in $C (A)$ with the reversed parameterisation $t \mapsto (x (-t), y (-t))$, labelled by $x^*$.
		\item In the remaining case, let $t$ be the parameter of the turning point.
		Add a new node at position $(x (t), y (t))$, labelled by the morphism symbol $\epsilon_x$.
		Add an edge from $(x (0), y (0))$ to the new node labelled by $x$, and from $(x (1), y (1))$ to the new node labelled by $x^*$.
		(The new edges can follow the same path as the original one, although if $x (t)$ is locally decreasing at $(x (t), y (t))$ we need to slightly alter the parameterisation, so that the $x$-labelled edge enters below and the $x^*$-labelled edge above, to match the type $\epsilon_x : x \otimes x^* \to I$.)
	\end{itemize}
\end{definition}

An example of this construction is illustrated in figure \ref{fig-circuit-construction}.

\begin{figure}
	\begin{center} \begin{tikzpicture}[auto]
		\node [trapezium, trapezium left angle=0, trapezium right angle=75, shape border rotate=90, trapezium stretches=true, minimum height=1cm, minimum width=2cm, draw] (f) at (0, 2.25) {$f$};
		\node [trapezium, trapezium left angle=75, trapezium right angle=0, shape border rotate=270, trapezium stretches=true, minimum height=1cm, minimum width=2cm, draw] (g) at (0, -.25) {$g$};
		\draw [->-] (f) to [out=0, in=90] (1.5, 1) to [out=-90, in=0] node {$x$} (g);
		\node at (3, 1) {$\mapsto$};
		\node [rectangle, minimum height=1.75cm, minimum width=1cm, draw] (f') at (5, 2) {$f$};
		\node [rectangle, minimum height=1.75cm, minimum width=1cm, draw] (g') at (5, 0) {$g^*$};
		\node [rectangle, minimum height=1.75cm, minimum width=1cm, draw] (e) at (7, 1) {$\epsilon_x$};
		\node (d1) at (0, 1.5) {}; \node (d2) at (0, .5) {};
		\draw [-] (f') to [out=0, in=180] node [above=5pt, near start] {$x$} (e.west |- d2);
		\draw [-] (g') to [out=0, in=180] node [below, near start] {$x^*$} (e.west |- d1);
	\end{tikzpicture} \end{center}
	\caption{Example construction of $C (A)$}
	\label{fig-circuit-construction}
\end{figure}
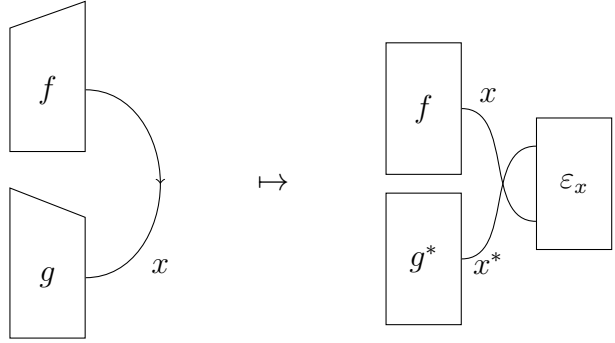

\begin{lemma}\label{lemma-c-isomorphisms}
	Let $A$ and $B$ be teleological diagrams.
	Then $A \cong_c B$ iff $C (A) \cong_c C (B)$.
\end{lemma}

Note that the former refers to isomorphism of teleological diagrams without rotations (definition \ref{def:teleological-circuit-isomorphism}), while the latter refers to isomorphism of circuit diagrams (definition \ref{def:circuit-diagram-isomorphism}).

\begin{proof}
	Given an isomorphism $i : A \cong_c B$ of teleological diagrams it is straightforward to construct an equivalence $C (i) : C (A) \cong_c C (B)$ of circuit diagrams, and conversely.
	The only interesting case is when there is an edge $e : \alpha^+ \to \beta^-$ in $A$ with a maximum at $(x (t), y (t))$.
	Since $i \not\reflects \alpha$ and $i \not\reflects \beta$ we also have $i (e) : i (\alpha)^+ \to i (\beta)^-$, and hence $i (e)$ also has a maximum at some $(x' (t'), y' (t'))$.
	These maxima correspond to fresh nodes in $C (A)$ and $C (B)$ respectively, and we construct $C (i)$ to associate them.
\end{proof}

\begin{proof}[of theorem \ref{main-theorem}]
	We begin by extending $w$ to a monoidal valuation $M (w) : M (\Sigma) \to U (\C)$ by setting $M (w) (x^*) = w (x)^*$, $M (w) (f^*) = w (f)^*$ and $M (w) (\epsilon_x) = \epsilon_{w (x)}$.
	By the Joyal-Street coherence theorem (theorem \ref{joyal-street-theorem}), we have a symmetric monoidal functor $G : \D_c (M (\Sigma)) \to \C$ such that $M (w) = G \circ v_{M (\Sigma)}$, unique up to unique monoidal natural isomorphism.
	
	We will prove that this $G$ is a teleological functor.
	More formally, recall from definition \ref{def-teleological-functor} that a teleological functor $F : \D_t (\Sigma) \to \C$ is a symmetric monoidal functor satisfying additional properties.
	We will prove that $G$ factors as
	\[ \D_c (M (\Sigma)) \overset{E}{\longrightarrow} \D_t (\Sigma) \overset{F}{\longrightarrow} \C \]
	and then that $F$ satisfies the conditions of a teleological functor.
	This amounts to proving that $G$ is defined on morphisms of $\D_t (\Sigma)$, which are $\cong_t$-equivalence classes, and that it respects duals and counits.
	Intuitively, this is possible because although $G$ is only a symmetric monoidal functor, the additional teleological structure is preserved by the valuation $v_{M (\Sigma)}$.
	
	First we prove that $G$ is defined on $\cong_t$-equivalence classes,
	Let $A$ and $B$ be teleological diagrams over $\Sigma$ with $A \cong_t B$.
	By lemma \ref{normal-form-lemma}, we have teleological diagrams $A' \cong_c A$ and $B' \cong_c B$ where $A'$ and $B'$ have the form depicted in figure \ref{fig:normal-form-lemma}.
	By lemma \ref{lemma-c-isomorphisms} we have $M (A') \cong_c M (A)$ and $M (B') \cong_c M (B)$.
	Since $G$ is defined on $\cong_c$-equivalence classes of circuit diagrams, it follows that $G (M (A')) = G (M (A))$ and $G (M (B')) = G (M (B))$.
	
	We next prove that $G (M (A')) = G (M (B'))$.
	Since $G$ is a symmetric monoidal functor, these expressions factorise as
	\[ G (M (A')) = (Y \otimes (\epsilon_{Y'} \circ (G (M (Q)) \otimes Y'{}^*))) \circ G (M (P)) : X \to Y \]
	\[ G (M (B')) = (Y \otimes (\epsilon_{X'} \circ (X' \otimes G (M (Q))^*))) \circ G (M (P)) : X \to Y \]
	where
	\begin{align*}
		X &= G (x_1)^{o_1} \otimes \cdots \otimes G (x_m)^{o_m} \\
		Y &= G (y_1)^{o'_1} \otimes \cdots \otimes G (y_n)^{o'_n} \\
		X' &= G (x'_1)^{o''_1} \otimes \cdots \otimes G (x'_{m'})^{o''_{m'}} \\
		Y' &= G (y'_1)^{o'''_1} \otimes \cdots \otimes G (y'_{n'})^{o'''_{n'}}
	\end{align*}
	using the notation of figure \ref{fig:normal-form-lemma}.
	The equality then follows from extranaturality of $\epsilon$:
	\begin{center} \begin{tikzpicture}[auto]
		\node (A) at (0, 3) {$G (X') \otimes G (Y')^*$};
		\node (B) at (8, 3) {$G (Y') \otimes G (Y')^*$};
		\node (C) at (0, 0) {$G (X') \otimes G (X')^*$};
		\node (D) at (8, 0) {$I$};
		\draw [->] (A) to node {$G (Q) \otimes G (Y')^*$} (B); \draw [->] (A) to node {$G (X') \otimes G (Q)^*$} (C);
		\draw [->] (B) to node {$\epsilon_{G (Y')}$} (D); \draw [->] (C) to node {$\epsilon_{G (X')}$} (D);
	\end{tikzpicture} \end{center}
	
	This completes the proof that $G$ is defined on $\cong_t$-equivalence classes.
	Consequently, we have defined a symmetric monoidal functor $F : \D_t (\Sigma) \to \C$ satisfying $w = F \circ v_\Sigma$.
	
	Next we must prove that remaining conditions of definition \ref{def-teleological-functor} to show that $F$ is a teleological functor.
	The first condition is that $F (X^*) = F (X)^*$, and similarly for dualisable morphisms.
	Since all objects of $\D_t (\Sigma)$ are generated by objects of the form $v_\Sigma (x)$, and $F$ is a monoidal functor, it suffices to prove it for those:
	\begin{align*}
		F (v_\Sigma (x)^*) &= G (v_{M (\Sigma)} (x^*)) &&\text{by definition of $F$} \\
		&= M (w) (x^*) &&\text{since $M (w) = G \circ v_{M (\Sigma)}$} \\
		&= w (x)^* &&\text{by definition of $M (w)$} \\
		&= F (v_\Sigma (x))^* &&\text{since $w = F \circ v_\Sigma$}
	\end{align*}
	Similarly, for diagrams of the form $v_\Sigma (f)$ for $f : \mor_d (\Sigma)$,
	\[ F (v_\Sigma (f)^*) = F (E (v_{M (\Sigma)} (f^*))) = G (v_{M (\Sigma)} (f^*)) = M (w) (f^*) = w (f)^* = F (v_\Sigma (f))^* \]
	
	In order to prove that $F (\epsilon_X) = \epsilon_{F (X)}$ for all objects $X$, we work by structural induction on $X$.
	If $X = v_\Sigma (x)$ then
	\begin{align*}
		F (\epsilon_{v_\Sigma (x)}) &= G (v_{M (\Sigma)} (\epsilon_x)) &&\text{by definition of $F$} \\
		&= M (w) (\epsilon_x) &&\text{since $M (w) = G \circ v_{M (\Sigma)}$} \\
		&= \epsilon_{w (x)} &&\text{by definition of $M (w)$} \\
		&= \epsilon_{F (v_\Sigma (x))} &&\text{since $w = F \circ v_\Sigma$}
	\end{align*}
	If $X = v_\Sigma (x)^*$ then
	\begin{align*}
		F (\epsilon_{v_\Sigma (x)^*}) &= F (\sigma_{v_\Sigma (x)^*, v_\Sigma (x)} \circ \epsilon_{v_\Sigma (x)}) &&\text{axiom of teleological categories} \\
		&= \sigma_{F (v_\Sigma (x)^*), F (v_\Sigma (x))} \circ F (\epsilon_{v_\Sigma (x)}) &&\text{since $F$ is a symmetric monoidal functor} \\
		&= \sigma_{F (v_\Sigma (x)^*), F (v_\Sigma (x))} \circ \epsilon_{F (v_\Sigma (x))} &&\text{by the previous case} \\
		&= \sigma_{F (v_\Sigma (x))^*, F (v_\Sigma (x))} \circ \epsilon_{F (v_\Sigma (x))} &&\text{since $F (v_\Sigma (x)^*) = F (v_\Sigma (x))^*$} \\
		&= \epsilon_{F (v_\Sigma (x))^*} &&\text{axiom of teleological categories} \\
		&= \epsilon_{F (v_\Sigma (x)^*)} &&\text{since $F (v_\Sigma (x)^*) = F (v_\Sigma (x))^*$}
	\end{align*}
	Similarly, if $X = X_1 \otimes X_2$ then
	\begin{align*}
		F (\epsilon_{X_1 \otimes X_2}) &= F ((X_1 \otimes \sigma_{X_1^*, X_2} \otimes X_2^*) \circ (\epsilon_{X_1} \otimes \epsilon_{X_2})) \\
		&= (F (X_1) \otimes \sigma_{F (X_1^*), F (X_2)} \otimes F (X_2^*)) \circ (F (\epsilon_{X_1}) \otimes F (\epsilon_{X_2})) \\
		&= (F (X_1) \otimes \sigma_{F (X_1)^*, F (X_2)} \otimes F (X_2)^*) \circ (\epsilon_{F (X_1)} \otimes \epsilon_{F (X_2)}) \\
		&= \epsilon_{F (X_1) \otimes F (X_2)} \\
		&= \epsilon_{F (X_1 \otimes X_2)}
	\end{align*}
	
	Finally we come to uniqueness.
	Let $F' : \D_t (\Sigma) \to \C$ be another teleological functor with $w = F' \circ v_\Sigma$.
	This extends to a symmetric monoidal functor
	\[ \D_c (M (\Sigma)) \overset{E}{\longrightarrow} \D_t (\Sigma) \xrightarrow{F'} \C \]
	which is equal to $G$ by uniqueness.
	Repeating the construction of $F$ from $G$, we conclude that $F = F'$.
\end{proof}

\bibliographystyle{plainurl}
\bibliography{\string~/Dropbox/Work/refs}

\begin{thebibliography}{10}

\bibitem{saleh_etal_notions_bidirectional_computation_entangled_state_monads}
Faris Abou-Saleh, James Cheney, Jeremy Gibbons, James McKinna, and Perdita
  Stevens.
\newblock Notions of bidirectional computation and entangled state monads.
\newblock In {\em Mathematics of program construction 2015}, volume 9129 of
  {\em Lecture Notes in Computer Science}. Springer, 2015.
\newblock \href {http://dx.doi.org/10.1007/978-3-319-19797-5_9}
  {\path{doi:10.1007/978-3-319-19797-5_9}}.

\bibitem{abramsky05}
Samson Abramsky.
\newblock Abstract scalars, loops, and free traced and strongly compact closed
  categories.
\newblock In {\em Proceedings of CALCO'05}, volume 3629 of {\em Lectures notes
  in computer science}, pages 1--29, 2005.
\newblock \href {http://dx.doi.org/10.1007/11548133_1}
  {\path{doi:10.1007/11548133_1}}.

\bibitem{chantawibul_sobocinski_towards_compositional_graph_theory}
Apiwat Chantawibul and Pawe{\l} Soboci\'nski.
\newblock Towards compositional graph theory.
\newblock In {\em Proceedings of {MFPS'15}}, volume 319 of {\em ENTCS}, pages
  121--136, 2015.
\newblock \href {http://dx.doi.org/10.1016/j.entcs.2015.12.009}
  {\path{doi:10.1016/j.entcs.2015.12.009}}.

\bibitem{coecke_kissinger_picturing_quantum_processes}
Bob Coecke and Aleks Kissinger.
\newblock {\em Picturing quantum processes}.
\newblock Cambridge University Press, 2017.

\bibitem{coecke_lal_causal_categories}
Bob Coecke and Raymond Lal.
\newblock Causal categories: {R}elativistically interacting processes.
\newblock {\em Foundations of physics}, 43(4):458--501, 2013.
\newblock \href {http://dx.doi.org/10.1007/s10701-012-9646-8}
  {\path{doi:10.1007/s10701-012-9646-8}}.

\bibitem{depaiva91a}
Valeria de~Paiva.
\newblock The dialectica categories.
\newblock Technical report, University of Cambridge, 1991.

\bibitem{sep-aristotle-causality}
Andrea Falcon.
\newblock Aristotle on causality. {I}n \textit{The Stanford Encyclopedia of
  Philosophy} ({S}pring 2015 edition) [online].
\newblock 2015.
\newblock URL:
  \url{https://plato.stanford.edu/archives/spr2015/entries/aristotle-causality/}.

\bibitem{fong_algebra_open_interconnected_systems}
Brendan Fong.
\newblock {\em The algebra of open and interconnected systems}.
\newblock PhD thesis, University of Oxford, 2016.

\bibitem{foster_etal_combinators_bidirectional_tree_transformations}
Nate Foster, Michael Greenwald, Jonathan Moore, Benjamin Pierce, and Alan
  Schmidt.
\newblock Combinators for bi-directional tree transformations: a linguistic
  approach to the view update problem.
\newblock {\em ACM Transactions on Programming Languages and Systems}, 29(3),
  2007.
\newblock \href {http://dx.doi.org/10.1145/1232420.1232424}
  {\path{doi:10.1145/1232420.1232424}}.

\bibitem{foster_three_complementary_approaches_bidirectional_programming}
Nate Foster, Kazutaka Matsuda, and Janis Voigtl\"ander.
\newblock Three complementary approaches to bidirectional programming.
\newblock In Jeremy Gibbons, editor, {\em Generic and indexed programming},
  volume 7470 of {\em Lecture Notes in Computer Science}, pages 1--42, 2012.

\bibitem{hedges15b}
Neil Ghani, Jules Hedges, Viktor Winschel, and Philipp Zahn.
\newblock A compositional approach to economic game theory.
\newblock arXiv:1603.04641, 2016.

\bibitem{goforth_robinson_topology_2x2_games}
David Goforth and David Robinson.
\newblock {\em The topology of the $2 \times 2$ games: {A} new periodic table}.
\newblock Routledge advances in game theory. Routledge, 2005.

\bibitem{hedges_towards_compositional_game_theory}
Jules Hedges.
\newblock {\em Towards compositional game theory}.
\newblock PhD thesis, Queen Mary University of London, 2016.

\bibitem{hofmann_etal_symmetric_lenses}
Martin Hofmann, Benjamin Pierce, and Daniel Wagner.
\newblock Symmetric lenses.
\newblock In {\em Proceedings of the 38th annual ACM SIGPLAN-SIGACT symposium
  on Principles of programming languages (POPL'11)}, volume~46, pages 371--384,
  2011.
\newblock \href {http://dx.doi.org/10.1145/1926385.1926428}
  {\path{doi:10.1145/1926385.1926428}}.

\bibitem{jacobs-categorical-logic-type-theory}
Bart Jacobs.
\newblock {\em Categorical logic and type theory}.
\newblock Studies in logic and the foundations of mathematics. Elsevier, 1999.

\bibitem{johnson_rosebrugh_spans_lenses}
Michael Johnson and Robert Rosebrugh.
\newblock Spans of lenses.
\newblock In {\em Proceedings of the Workshops of the EDBT/ICDT 2014 Joint
  Conference}, volume 1133 of {\em CEUR Workshop Proceedings}, pages 112--118,
  2014.

\bibitem{johnson_etal_lenses_fibrations_universal_transformations}
Michael Johnson, Robert Rosebrugh, and R.~J. Wood.
\newblock Lenses, fibrations and universal transformations.
\newblock {\em Mathematical structures in computer science}, 22(1):25--42,
  2011.
\newblock \href {http://dx.doi.org/10.1017/S0960129511000442}
  {\path{doi:10.1017/S0960129511000442}}.

\bibitem{joyal_street_planar_diagrams_tensor_algebra}
Andr\'e Joyal and Ross Street.
\newblock Planar diagrams and tensor algebra.
\newblock Unpublished manuscript, 1988.

\bibitem{joyal91}
Andr\'e Joyal and Ross Street.
\newblock The geometry of tensor calculus {I}.
\newblock {\em Advances in mathematics}, 88(1):55--112, 1991.
\newblock \href {http://dx.doi.org/10.1016/0001-8708(91)90003-P}
  {\path{doi:10.1016/0001-8708(91)90003-P}}.

\bibitem{kazak-lenses-over-tea}
Artyom Kazak.
\newblock Lenses over tea [online].
\newblock URL: \url{https://artyom.me/#lens-over-tea}.

\bibitem{kelly80}
G.~M. Kelly and M.~L. Laplaza.
\newblock Coherence for compact closed categories.
\newblock {\em Journal of pure and applied algebra}, 19(193--213), 1980.
\newblock \href {http://dx.doi.org/10.1016/0022-4049(80)90101-2}
  {\path{doi:10.1016/0022-4049(80)90101-2}}.

\bibitem{mellies_functorial_boxes_string_diagrams}
Paul-Andr\'e Melli\`es.
\newblock Functorial boxes in string diagrams.
\newblock In {\em Proceedings of {CSL} 2006}, Lecture Notes in Computer
  Science, pages 1--30. Springer, 2006.
\newblock \href {http://dx.doi.org/10.1007/11874683_1}
  {\path{doi:10.1007/11874683_1}}.

\bibitem{pavlovic_geometry_abstraction_quantum_computation}
Dusko Pavlovic.
\newblock Geometry of abstraction in quantum computation.
\newblock {\em Proceedings of Symposia in Applied Mathematics}, 71:233--267,
  2012.

\bibitem{pickering_gibbons_wu_profunctor_optics}
Matthew Pickering, Jeremy Gibbons, and Nicolas Wu.
\newblock Profunctor optics: {M}odular data accessors.
\newblock {\em The art, science and engineering of programming}, 1(2), 2017.
\newblock \href {http://dx.doi.org/10.22152/programming-journal.org/2017/1/7}
  {\path{doi:10.22152/programming-journal.org/2017/1/7}}.

\bibitem{selinger11}
Peter Selinger.
\newblock A survey of graphical languages for monoidal categories.
\newblock In Bob Coecke, editor, {\em New structures for physics}, pages
  289--355. Springer, 2011.

\bibitem{shulman_framed_bicategories_monoidal_fibrations}
Michael Shulman.
\newblock Framed bicategories and monoidal fibrations.
\newblock {\em Theory and applications of categories}, 20(18):650--738, 2008.

\bibitem{vonneumann44}
John von Neumann and Oskar Morgenstern.
\newblock {\em Theory of games and economic behaviour}.
\newblock Princeton university press, 1944.

\end{thebibliography}

\end{document}